\documentclass[a4paper,12pt]{amsart}

\usepackage{amsmath,amsfonts,amsthm,amssymb}
\usepackage{fullpage}
\usepackage{xcolor}
\usepackage{hyperref}
\usepackage[all]{xy}

\swapnumbers 
\theoremstyle{plain}    
\newtheorem{theorem}[subsection]{Theorem}
\newtheorem{lemma}[subsection]{Lemma}
\newtheorem{proposition}[subsection]{Proposition}
\newtheorem{corollary}[subsection]{Corollary}

\theoremstyle{definition}   
\newtheorem{definition}[subsection]{Definition}

\theoremstyle{remark}   
\newtheorem{remark}[subsection]{Remark}

\newcommand{\Int}{\mathbb{Z}}   
\newcommand{\Q}{\mathbb{Q}}     

\newcommand{\Derived}{\mathbf{D}}   
\newcommand{\modules}{\textendash\mathrm{mod}} 

\newcommand{\Hom}{\text{Hom}}   
\newcommand{\ext}{\text{Ext}}   
\newcommand{\tor}{\text{Tor}}   
\newcommand{\coker}{\mathrm{coker}} 
\newcommand{\cone}{\mathbf{C}}      
\newcommand{\chains}{\mathrm{C}}    
\newcommand{\holim}{\mathrm{holim}} 
\newcommand{\colim}{\mathrm{colim}} 
\newcommand{\op}{\mathrm{op}}       
\newcommand{\spec}{\mathrm{Spec}}   
\newcommand{\supp}{\mathrm{supp}}   

\newcommand{\shr}{\chi}             
\newcommand{\aug}{\mathfrak{a}}     
\newcommand{\ba}{\boxtimes}         
\newcommand{\udl}{\natural}         
\newcommand{\A}{\mathcal{A}}

\newcommand{\F}{\mathcal{F}}
\newcommand{\G}{\mathcal{G}}

\newcommand{\uu}{\mathcal{U}}       
\newcommand{\vv}{\mathcal{V}}       
\newcommand{\pp}{\mathfrak{p}}      
\newcommand{\qq}{\mathfrak{q}}      
\newcommand{\cl}{\mathrm{cl}}      

\newcounter{commentcount}[section]

\begin{document}

\title{A spectral sequence for the Hochschild cohomology of a coconnective dga}
\author{S. Shamir}
\address{Department of Mathematics, University of Bergen, 5008 Bergen, Norway}
\email{shoham.shamir@math.uib.no}
\date{\today}
\thanks{The author was supported by the RCN ``Topology'' project.}

\begin{abstract}
A spectral sequence for the computation of the Hochschild cohomology of a coconnective dga over a field is presented. This spectral sequence has a similar flavour to the spectral sequence presented in \cite{CohenJonesYan} for the computation of the loop homology of a closed orientable manifold. Using this spectral sequence we identify a class of spaces for which the Hochschild cohomology of their mod-$p$ cochain algebra is Noetherian. This implies, among other things, that for such a space the derived category of mod-$p$ chains on its loop-space carries a theory of support varieties.
\end{abstract}

\maketitle                 

\section{Introduction}

Given a differential graded algebra (a \emph{dga}) $A$ over a field $k$, the Hochschild cohomology of $A$, with coefficients in the differential graded $A$-bimodule $A$, is
\[ HH^*(A) = \ext^*_{A \otimes_k A^\op} (A,A) . \]
For an augmented $k$-dga $A$ there is a well known map of graded algebras $\shr: HH^*(A) \to \ext^*_A(k,k)$, which we shall call the \emph{shearing map}. Several equivalent definitions of the shearing map are given in Section~\ref{sec: The Hochschild cohomology shearing map}. The main result of this paper, Theorem~\ref{thm: Main theorem}, describes a multiplicative spectral sequence for the computation of $HH^*(A)$, which also computes the image of the shearing map. The existence of such a spectral sequence is hardly surprising, and so we must explain why the consequences are interesting.

The main consequence of this spectral sequence is Theorem~\ref{thm: Elliptic mod p dgas}, which gives conditions under which $HH^*(A)$ is Noetherian and $\ext^*_A(k,k)$ is finitely generated over the image of $\shr$. Our interest in this result comes from its application to topological spaces. Given a simply connected space $X$, we consider its cochains dga $\chains^*(X;k)$, i.e. the singular cochains on $X$ with coefficients in $k$. There is a particular class of spaces for whose cochain dga Theorem~\ref{thm: Elliptic mod p dgas} holds.

\begin{definition}
A space $X$ is called \emph{$p$-finite}, for a prime $p$, if $H_*(X;\Int/p)$ is finite dimensional. Recall from~\cite{FelixHalperinThomasEliiptic1} that a simply connected space $X$ is called \emph{$p$-elliptic} if it is of finite type and the sequence $\{\dim H_n(\Omega X;\Int/p)\}_n$ grows at most at a polynomial rate. Compact Lie groups and homogeneous spaces are noteworthy examples of elliptic spaces.
\end{definition}

For simplicity of exposition we shall now fix a prime $p$ and let $k=\Int/p$. Recall, that for any $p$-finite simply connected space $X$ there is a well known isomorphism of graded algebras $H_*(\Omega X;k) \cong \ext^{-*}_{\chains^*(X;k)}(k,k)$. Theorem~\ref{thm: Elliptic mod p dgas} implies the following result:

\begin{corollary}
\label{cor: p-finite p-elliptic spaces}
Let $X$ be a $p$-finite $p$-elliptic space. Then $HH^*(\chains^*(X;k))$ is Noetherian and the shearing map makes $H_*(\Omega X;k)$ into a finitely generated $HH^*(\chains^*(X;k))$-module.
\end{corollary}

Note that, in the rational case, it is very easy to construct a $\Q$-finite $\Q$-elliptic space $X$ where $H_*(\Omega X;\Q)$ is not finitely generated over the image of the shearing map, see \cite[Example 8.5]{GreenleesHessShamir}.

We are interested in spaces with Noetherian Hochschild cohomology for two reasons, one coming from commutative algebra and the other coming from non-commutative algebra. In commutative algebra, the obvious analogue to $p$-elliptic spaces are local complete intersection algebras. Among Noetherian commutative local algebras the \emph{complete intersection (ci)} local algebras are characterized as those algebras $R$ for which the sequence $\{\dim \ext_R^n(k,k)\}_n$ has polynomial growth, where $k$ is the residue field of $R$. The analogy between ci local algebras and cochains on $p$-finite $p$-elliptic spaces is strengthened by the fact that finite dimensional ci local algebras have Noetherian Hochschild cohomology. As far as the author is aware, it is still an open question whether having Noetherian Hochschild cohomology characterizes such ci local algebras.

This analogy between cochains on $p$-elliptic spaces and ci local algebras is the main focus of~\cite{BensonGreenleesShamir}. There, the analogy is explained in detail, and other properties of ci local algebras are shown to hold for cochains on $p$-elliptic spaces. We shall not say more about this except to note that Corollary~\ref{cor: p-finite p-elliptic spaces} is central to the arguments in~\cite{BensonGreenleesShamir}.

Recent work in non-commutative algebra motivates the use of Hochschild cohomology for development of support varieties. Given a non-commutative algebra $\Lambda$, one would like a notion of support for $\Lambda$-modules, similar to that of modules over a commutative ring. It has been observed that the Hochschild cohomology of $\Lambda$ can be used for that purpose, see~\cite{Solberg} for a survey on this subject. Such a theory of support is easiest to construct when the Hochschild cohomology is Noetherian. But it is difficult to determine when this condition holds; for example, in~\cite{Xu} Xu constructs a 7-dimensional algebra whose Hochschild cohomology ring is not Noetherian, even modulo nilpotents.

On the topology side, we consider the dga $\chains_*(\Omega X;k)$ where $\Omega X$ is the loop-group of $X$. The derived category of $\chains_*(\Omega X;k)$ contains much information about the space $X$ (see~\cite{BlumbergCohenTeleman} for a recent motivating example). For a $p$-finite $p$-elliptic space we shall see that the Hochschild cohomology of $\chains_*(\Omega X;k)$ is Noetherian. Once this has been established, modern technology developed by Benson, Iyengar and Krause~\cite{BIKsupport}, immediately gives the construction of support varieties for $\chains_*(\Omega X;k)$-modules, where the support of a module is a subset of the prime ideal spectrum of the Hochschild cohomology. For completeness, the definition and main properties of this theory of support is described in~\ref{sub: Support varieties for elliptic spaces} below.

We should note that our ultimate goal is to identify cases in which this notion of support classifies all localizing subcategories of the derived category $\Derived(\chains_*(\Omega X;k))$. In such cases, one says $HH^*(\chains^*(X;k))$ \emph{stratifies} $\Derived(\chains_*(\Omega X;k))$; this definition is also due to Benson, Iyengar and Krause~\cite{BIKstratifying}. Corollary~\ref{cor: p-finite p-elliptic spaces} can be viewed as a first step in this direction.

In order to describe the spectral sequence we must first give a few definitions. A differential graded algebra $A$ over a field $k$ will be called \emph{coconnective} if $A^n=0$ for $n<0$ and $H^0(A)=k$. It is of \emph{finite type} if $H^n(A)$ is a finite dimensional $k$-vector space for every $n$ and it is \emph{bounded} if $H^n(A)\neq 0$ only for a finite number of values of $n$. We say $A$ is \emph{simply connected} if $A$ is coconnective, augmented and $\ext^0_A(k,k)=k$ (equivalently $\tor^A_0(k,k)=k$). The dga $A \otimes_k A^\op$ will be denoted by $A^e$. Note there is a well known isomorphism $\ext^*_A(k,k) \cong \ext_{A^e}^*(A,k)$, which will be described in Section~\ref{sec: The Hochschild cohomology shearing map}.

Since we are computing the Hochschild cohomology of a coconnective dga, we are bound to end up with elements in both positive and negative degrees. We keep to the convention that subscript grading denotes homological degree, while superscript grading is cohomological, thus $X_\square = X^{-\square}$ for any graded object $X$. See also~\ref{sub: Notation and terminology} below.

\begin{theorem}
\label{thm: Main theorem}
Let $A$ be a simply connected dga over a field $k$. Then there exists a conditionally convergent multiplicative spectral sequence
\[ E^2_{p,q} =\ext^{-q}_{A^e}(A,H^{-p}(A)) \ \Longrightarrow \ HH^{-p-q}(A) . \]
Under the isomorphism $\ext^*_A(k,k) \cong \ext_{A^e}^*(A,k)$ the infinite cycles in $E^2_{0,*}$ can be identified with the image of the shearing map $\shr: HH^*(A) \to \ext^*_A(k,k)$. When $A$ is bounded the spectral sequence has strong convergence. When $A$ is of finite type then this spectral sequence takes the form
\[ E^2_{p,q} = H^{-p}(A) \otimes_k \ext_A^{-q}(k,k) \ \Longrightarrow \ HH^{-p-q}(A) , \]
with the obvious multiplicative structure on $H^{-p}(A) \otimes_k \ext_A^{-q}(k,k)$.
\end{theorem}

Several remarks are in order. First, a multiplicative spectral sequence means there is a multiplication defined on each $E^r$-term for which the differential $d^r$ is a derivation, the multiplication on $E^{r+1}$ is the one induced from $E^r$ in the obvious way and the resulting multiplication on $E^\infty$ agrees with the multiplication on the associated graded object of $HH^*(A)$. Note that we use the convergence conditions for spectral sequences as defined by Boardman in \cite{BoardmanCCSpecSeq}.

Second, we did not specify the multiplicative structure on the $E^2$-term of this spectral sequence in the case where $A$ is not of finite type. Roughly speaking, one can consider $A$ as a coalgebra with respect to the derived tensor product $\otimes_A^\mathbf{L}$. On the other side we have an appropriate pairing
$H^n(A) \otimes_A^\mathbf{L} H^m(A) \to H^{n+m}(A)$, induced by the multiplication on $A$. Together, these yield the multiplicative structure on the $E^2$-term. The precise description of this multiplication is given in Section~\ref{sec: The multiplicative structure}. When $A$ is of finite type both multiplicative structures on the $E^2$-term agree under the appropriate isomorphism, this is done in Lemma~\ref{lem: Identifying the E2-term}.

Last, the grading of this spectral sequence is homological. Hence, the differential on the $E^r$-term is $d^r:E^r_{p,q} \to E^r_{p-r,q+r-1}$, and the spectral sequence lies in the second quadrant. This has the unfortunate consequence of yielding minus signs on the $E^2$-term description. The choice of homological grading was motivated by topological examples, discussed below in~\ref{sub: Relation to other work}.

The assumption that a coconnective augmented dga $A$ is simply connected is satisfied in many cases. As is well known, when $H^1(A)=0$ then $A$ is simply connected, and for dgas of finite type this is also a necessary condition. The spectral sequence exists also when the dga $A$ is not simply connected, but in that case we can only describe the $E^1$-term of the spectral sequence.
\begin{proposition}
\label{lem: General spectral sequence}
Let $A$ be a coconnective augmented dga over a field $k$. Then there exists a conditionally convergent multiplicative spectral sequence
\[ E^1_{p,q} =\ext^{-q}_{A^e}(A,A^{-p}) \ \Longrightarrow \ HH^{-p-q}(A) .\]
\end{proposition}

The main consequence of this spectral sequence is the following result.
\begin{theorem}
\label{thm: Elliptic mod p dgas}
Let $k$ be a field of characteristic $p>0$. Let $A$ be a simply connected dga over $k$ such that $H^*(A)$ is finite dimensional and $\ext^*_A(k,k)$ is finitely generated over a central Noetherian sub-algebra. Then $\ext^*_A(k,k)$ is finitely generated over the image of the shearing map and $HH^*(A)$ is Noetherian.
\end{theorem}

\subsection{Support varieties for elliptic spaces}
\label{sub: Support varieties for elliptic spaces}

For completeness, we briefly recall here the results of~\cite{BIKsupport} which, by Corollary~\ref{cor: p-finite p-elliptic spaces}, are applicable for $p$-finite $p$-elliptic spaces.

Let $R$ be a graded-commutative Noetherian ring. Denote by $\spec R$ the homogeneous prime ideal spectrum of $R$. The \emph{specialization closure} of a subset $\uu \subset \spec R$ is the set
\[ \cl (\uu) = \{ \pp \in \spec R \ | \ \text{there exists } \qq\in \uu \text{ with } \qq \subset \pp\} .\]
A subset $\vv \subset \spec R$ is \emph{specialization closed} if $\cl (\vv) = \vv$.

Both the definition and the properties of support are summed up in the following corollary, which is mainly the application of~\cite[Theorem 1]{BIKsupport} to our setting. There are other useful properties that arise from the existence of a theory of support which are not listed here, these can be found in~\cite{BIKsupport}.

\begin{corollary}
\label{cor: Support}
Let $X$ be a $p$-finite $p$-elliptic space and let $R$ denote $HH^*(\chains^*(X;k))$. Then there exists a unique assignment sending each object $M \in \Derived(\chains_*(\Omega X;k))$ to a subset $\supp_R M$ of $\spec R$ satisfying:
\begin{enumerate}
\item For every $M \in \Derived(\chains_*(\Omega X;k))$
    \[\cl(\supp_R M) = \cl(\supp_R H_*(M)) ,\]
    where $\supp_R H_*(M)$ is the usual support of a module over a commutative ring.
\item For every $M,N \in \Derived(\chains_*(\Omega X;k))$
    \[\cl(\supp_R M) \cap \supp_R N = \emptyset \quad \text{implies} \quad \ext_{\chains_*(\Omega X;k)}^*(M,N) =0 .\]
\item For every exact triangle $M_1 \to M_2 \to M_3$ we have
    \[ \supp_R M_2 \subseteq \supp_R M_1 \cup \supp_R M_3 .\]
\item For any specialization closed subset $\vv \subset \spec R$ and for every $M \in \Derived(\chains_*(\Omega X;k))$ there is a distinguished triangle $M' \to M \to M''$ such that
    \[ \supp_R M' \subset \vv \quad \text{ and } \quad \supp_R M'' \subset \spec R \setminus \vv . \]
\end{enumerate}
In addition we have the following two properties:
\begin{itemize}
\item For every $M \in \Derived(\chains_*(\Omega X;k))$, $\supp_R M = \emptyset$ if and only if $M=0$.
\item If $M$ is a compact object of $\Derived(\chains_*(\Omega X;k))$ then $\supp_R M$ is the specialization closed subset $\supp_R H_*(M)$.
\end{itemize}
\end{corollary}

We should explain why $H_*(M)$ is an $R$-module in the corollary above. If $A$ is a $k$-dga, and $U$ and $V$ are $A$-modules, then $\ext^*_A(U,V)$ is known to be an $HH^*(A)$-module. In particular, $H_{-*}(V)=\ext^{*}_A(A,V)$ is an $HH^*(A)$-module. The proof of Corollary~\ref{cor: Support} relies on the fact that $HH^*(\chains^*(X;k)) \cong HH^*(\chains_*(\Omega X;k))$, which explains the $R$-action on $H_*(M)$.

\subsection{Relation to other work}
\label{sub: Relation to other work}
As noted in the abstract, the spectral sequence presented here bears some resemblance the spectral sequence for string homology of Cohen, Jones and Yan presented in \cite{CohenJonesYan}, which we now describe. Let $M$ be a closed orientable manifold, let $k$ be a commutative ring and let $A$ be the dga of singular cochains with coefficients in $k$. The $E^2$-term of the spectral sequence of~\cite{CohenJonesYan} is $E^2_{p,q} \cong H^{-p}(M,H_q(\Omega M))$, where $\Omega M$ is the based loop space of $M$. Assuming also that $M$ is simply connected then this $E^2$-term is isomorphic to $H^{-p}(A) \otimes_k \ext^{-q}_A(k,k)$.

The spectral sequence of \cite{CohenJonesYan} converges to the homology of the free (unbased) loops on $M$, properly desuspended, which they call the \emph{loop homology}. By results of Cohen and Jones \cite{CohenJonesStringTop}, this loop homology is isomorphic to the Hochschild cohomology of $A$, this is an isomorphism of graded algebras when the loop homology is given the Chas-Sullivan product. Thus, when $M$ is simply connected, both the spectral sequence given here and that of \cite{CohenJonesYan} have the same $E^2$ and $E^\infty$ terms and both are multiplicative. 

In \cite{FelixThomasVPHochschild}, Felix, Thomas and Vigu{\'e}-Poirrier consider a map $I:HH^*(A) \to \ext^*_{A^e}(A,k)$, where $A$ is the dga of singular cochains with coefficients in $k$ on a simply connected closed oriented manifold. They give a model for $I$, and using this model get several results concerning the kernel and image of $I$. In Lemma~\ref{lem: Sheering map as in FTVP paper} it is shown that $I$ is equal to the shearing map and using the spectral sequence presented here we recover one of their results.

\subsection{On the choice of setting and method of proof}
In \cite{DuggerMultiplicative1} and \cite{DuggerMultiplicative2}, Dugger gives a systematic treatment of the construction of multiplicative spectral sequences for topological spaces and for spectra. To mimic Dugger's work, we have have found it easier to use Quillen model category machinery. This also highlights the only significant difference between the construction here and Dugger's treatment - in this paper we are forced to use the bar construction, since we cannot assume that our filtration consists of cofibrant objects.

Roughly speaking, one gets a multiplicative structure on maps from a comonoid to a monoid in any monoidal category. In \cite{DuggerMultiplicative2}, Dugger shows how an appropriate filtration of the comonoid yields a multiplicative spectral sequence. Here we filter the monoid as Dugger does in \cite{DuggerMultiplicative1}. Viewed in this way, the construction of the spectral sequence is a simple translation of classical constructions from topology.

\subsection{Organization of this paper}
We start by presenting and proving the consequences of the spectral sequence in Section~\ref{sec: Consequences of the spectral sequence}. In Sections~\ref{sec: Model category preliminaries} and \ref{sec: Differential graded preliminaries} we set the necessary model category structure and the differential graded tools we will use. In Section~\ref{sec: Construction of the spectral sequence} the spectral sequence is constructed. Section~\ref{sec: The multiplicative structure} establishes the multiplicative properties of the spectral sequence. Section~\ref{sec: The Hochschild cohomology shearing map} gives various descriptions of the shearing map. Finally, in Section~\ref{sec: Identifying the $E^2$-term}, we identify the $E^2$-term of the spectral sequence.

The proof of Theorem~\ref{thm: Main theorem} is spread throughout this paper. Lemma~\ref{lem: Conditional convergence} shows the existence of the spectral sequence and its convergence properties. Proposition~\ref{pro: Multiplicativity of the HH spectral sequence} gives the multiplicative property. Lemma~\ref{lem: Identifying the E2-term} identifies the $E^2$-term of the spectral sequence when $A$ is of finite type. Finally, Lemma~\ref{lem: Infinite cycles} proves the statement regarding the image of the shearing map.

\subsection{Notation and terminology}
\label{sub: Notation and terminology}
A chain complex $X$ is described by a pair $(X^\udl,d^X)$ where $X^\udl$ is the underlying graded abelian group and $d^X$ is the differential. Note that we adhere to the convention that subscript grading is homological degree, while superscript grading is cohomological. Thus $X_n=X^{-n}$ and the differential lowers the homological degree (or raises the cohomological degree) $d:X_n \to X_{n-1}$. The tensor product of two chain complexes over $k$ will be denoted simply by $\otimes$.

Throughout $k$ is a field and $A$ is a coconnective augmented dga over $k$ whose augmentation map $A \to k$ is denoted by $\aug$. The opposite dga is denoted by $A^\op$ and $A^e$ is the dga $A \otimes A^\op$. The derived category of differential graded left $A$-modules will be denoted by $\Derived(A)$. By an \emph{equivalence} of chain complexes we mean a quasi isomorphism, as usual such morphisms are denoted by $\sim$.

We will refer to morphisms of chain complexes and differential graded $A$-modules as \emph{maps} (the relevant categories are described in Section~\ref{sec: Model category preliminaries}). This will serve to distinguish maps from the morphisms in the corresponding derived categories.

\subsection{Acknowledgments}
I am grateful to M.Brun for many useful discussions and for pointing out several mistakes in previous versions of this paper.

\section{Consequences of the spectral sequence and proofs of the main results}
\label{sec: Consequences of the spectral sequence}

\subsection{Noetherian property and the shearing map}
Three equivalent definitions of the shearing map will be given in Section~\ref{sec: The Hochschild cohomology shearing map}, but we give a quick review of one here. There is a way to assign to each element $x$ in the Hochschild cohomology $HH^*(A)$ a natural transformation $\zeta(x):1_{\Derived(A)} \to \Sigma^n 1_{\Derived(A)}$, see~\ref{sub: The third description of the shearing map} for details. Roughly speaking, given a morphism $x: A \to \Sigma^n A$ in $\Derived(A^e)$ we have the natural transformation $\zeta(x)_M=x \otimes_A^\mathbf{L} M$. This assignment preserves addition and turns multiplication to composition of natural transformations. We can now define the shearing map $\shr:HH^*(A) \to \ext^*_A(k,k)$ by $\shr(x)=\zeta(x)_k$.

>From the definition it is immediate that that the image of the shearing map $\shr$ lies in the centre of $\ext_A^*(k,k)$. In Lemma~\ref{lem: Sheering map as in FTVP paper} we show that our definition of the shearing map agrees with that of the morphism $I$ from \cite{FelixThomasVPHochschild}, thereby recovering their result that the image of $I$ is central. We also recover the following result of \cite[4.1 - Theorem 7]{FelixThomasVPHochschild}
\begin{theorem}
Let $A$ be a simply connected bounded dga over a field $k$ such that $H^n(A)=0$ for $n>d$. Then the kernel of the shearing map is nilpotent of nilpotency index less than or equal to $d$. If, in addition, $H^1(A)=0$ then the nilpotency index is less than or equal to $d/2$.
\end{theorem}
\begin{proof}
Since $A$ is bounded the spectral sequence has strong convergence. It also implies that all the elements in $E^r_{p,q}$ are nilpotent whenever $p\neq 0$. Theorem~\ref{thm: Main theorem} implies that $\oplus_{p<0}E^\infty_{p,*}$ is isomorphic to the kernel of the shearing map. To be precise, $E^\infty$ is isomorphic to the associated graded object of $HH^*(A)$ coming from some filtration, and under this isomorphism $\oplus_{p<0}E^\infty_{p,*}$ can be identified with the kernel of the shearing map.

Consequently, this means that the kernel of the shearing map is nilpotent and has nilpotency index which is less than or equal to the nilpotency index of the ideal $\oplus_{p<0}E^\infty_{p,*}\subset E^2_{*,*}$.
\end{proof}

As noted earlier, the main consequence of the spectral sequence is Theorem~\ref{thm: Elliptic mod p dgas} which we now prove.
\begin{proof}[Proof of Theorem~\ref{thm: Elliptic mod p dgas}]
Suppose $A$ is simply connected and bounded. Denote by $B$ the graded algebra $\ext_A^*(k,k)$. We are also assuming $B$ is finitely generated as a module over a central Noetherian sub-algebra $N$. We mean \emph{central} in the graded commutative sense, thus if $n \in N$ and $x \in B$ then $nx=(-1)^{|n||x|} xn$, but this will play no part in what follows.

Since $N$ is Noetherian then by Noether Normalization it contains a polynomial sub-algebra $P=k[x_1,...,x_n]$ such that $N$ is finitely generated over $P$. Thus the degree of each $x_i$ must be even. By identifying $B$ with $E^2_{0,*}$ we shall now consider $P$ as a sub-algebra of $E^2$. By the description of the multiplication on the $E^2$-term when $A$ is of finite type we see that $P$ is central in $E^2$.

We now employ an argument of Quillen. By the Leibnitz rule and the fact that $x_i$ is central and of even degree, we see that $x_i^p$ is a cycle for $d^2$, for every $i$. Continuing in this fashion, we see that $x_i^{p^r}$ is a cycle for the differential $d^r$ on $E^r$. Since $A$ is bounded the spectral sequence collapses at some finite stage, say $R$. Hence $x_1^{p^{R}},...,x_n^{p^R}$ are all infinite cycles and so in the image of the shearing map. Clearly $B$ is finitely generated as a module over the sub-algebra $P'=k[x_1^{p^{R}},...,x_n^{p^R}]$.

Since $H^*(A)$ is finite dimensional, we see that $E^2$ is finitely generated as a module over $P'$. Hence $E^2$ is a Noetherian $P'$-module. In addition, because the elements of $P'$ are infinite cycles we see, by the Leibnitz rule, that the differential $d^2$ is a morphism of $P'$-modules. Hence the kernel of $d^2$ is a Noetherian $P'$-module and so is this kernel's quotient $E^3$. Continuing by induction we end up showing that $E^R=E^\infty$ is a Noetherian $P'$-module. This shows that $HH^*(A)$ is finitely generated as a module over a polynomial sub-algebra and hence Noetherian.
\end{proof}

\begin{proof}[Proof of Corollary~\ref{cor: p-finite p-elliptic spaces}]
The results of Felix, Halperin and Thomas from~\cite{FelixHalperinThomasHopfAlgebraElliptic} and~\cite{FelixHalperinThomasEliiptic1} immediately imply that for any $p$-finite $p$-elliptic space $X$ the dga $\chains^*(X;\Int/p)$ satisfies the conditions of Theorem~\ref{thm: Elliptic mod p dgas}.
\end{proof}

\subsection{Constructing support varieties}

\begin{proof}[Proof of Corollary~\ref{cor: Support}]
It follows from the results of Felix, Menichi and Thomas~\cite{FelixMenichiThomas} that there is an isomorphism
\[ HH^*(\chains^*(X;k)) \cong HH^*(\chains_*(\Omega X;k)) .\]
We also note that under this isomorphism the shearing map becomes the obvious morphism
\[HH^{-*}(\chains_*(\Omega X;k)) \to H_*(\Omega X;k) .\]
>From Corollary~\ref{cor: p-finite p-elliptic spaces} we see that $R=HH^*(\chains_*(\Omega X;k))$ is Noetherian and therefore the machinery of Benson, Iyengar and Krause~\cite{BIKsupport} applies.

Only the last property requires additional attention. Let $M$ be a compact object of $\Derived(\chains_*(\Omega X;k))$. Since $X$ is $p$-elliptic we see that $H_*(\Omega X;k)$ is Noetherian. By Theorem~\ref{thm: Elliptic mod p dgas}, $H_*(\Omega X;k)$ is finitely generated over $R$. We conclude that both $\ext_{\chains^*(\Omega X;k)}^*(M,M)$ and $H_*(M)$ are finitely generated over $R$. The final property now follows from~\cite[Theorem 5.4]{BIKsupport}.
\end{proof}

The next proposition is only a simple observation; it is given in order to show that the support is well behaved. For a fixed space $X$ recall that the category of spaces under $X$ has for objects maps $f:X \to Y$ and morphisms are commuting triangles. We consider the subsets of a given set as a category with morphisms being inclusions.

\begin{proposition}
\label{pro: Functor from spaces under to support}
Let $X$ be a $p$-finite $p$-elliptic pointed space and let $k=\Int/p$. Then the support defined in Corollary~\ref{cor: Support} induces a contravariant functor
\[ \Psi:\left\{
     \begin{array}{c}
        \text{Pointed connected}\\
        \text{spaces under }X\\
     \end{array}
   \right\}^\op
   \longrightarrow
   \Big\{\text{subsets of } \spec\, HH^*(\chains^*(X;k)) \Big\}
\]
where $\Psi(f:X \to Y) = \supp_{HH^*(\chains^*(X;k))} \chains_*(\Omega Y;k)$ and $\chains_*(\Omega Y;k)$ is a $\chains_*(\Omega X;k)$-module via the map $\chains_*(\Omega f;k)$.
\end{proposition}

Note that for spaces $Y$ under $X$ such that $\chains_*(\Omega Y;k)$ is compact in $\Derived( \chains_*(\Omega X;k))$ the proposition is obvious, since in this case the support is simply the support of $H_*(\Omega Y;k)$.

To prove Proposition~\ref{pro: Functor from spaces under to support} we must understand a little more about the construction of support. For the rest of this section we fix a prime $p$, set $k=\Int/p$, let $\Derived$ denote the derived category of $\chains_*(\Omega X;k)$ and let $R=HH^*(\chains_*(\Omega X;k))$.

For every prime ideal $\pp\in \spec R$, Benson, Iyengar and Krause \cite{BIKsupport} construct a triangulated functor $\Gamma_\pp:\Derived \to \Derived$ which preserves coproducts. The support of an object $M\in \Derived$ is defined by
\[ \supp_R M = \{ \pp \in \spec R \ | \ \Gamma_\pp M \neq 0\}.\]
Recall that a \emph{localizing subcategory} of $\Derived$ is a full triangulated subcategory closed under coproducts. The \emph{localizing subcategory generated by an object} is the minimal localizing subcategory containing that object. Since $\Gamma_\pp$ preserves coproducts then the inverse image under $\Gamma_\pp$ of a localizing subcategory is a localizing subcategory.

\begin{proof}[Proof of Proposition~\ref{pro: Functor from spaces under to support}]
Let $f:X \to Y$ and $g:X \to Z$ be pointed connected spaces under $X$ and let $h:Y \to Z$ be a map of spaces under $X$. We must show that
\[ \supp_R \chains_*(\Omega Z;k) \subset \supp_R \chains_*(\Omega Y;k).\]

It is a simple exercise to show that $\chains_*(\Omega Z;k)$ lies in the localizing subcategory of $\Derived$ generated by $\chains_*(\Omega Y;k)$. It follows that for every prime $\pp \in \spec R$ the object $\Gamma_\pp  \chains_*(\Omega Z;k)$ is contained in the localizing subcategory generated by $\Gamma_\pp  \chains_*(\Omega Y;k)$.
\end{proof}

\section{Model category preliminaries}
\label{sec: Model category preliminaries}

\subsection{The category of dg-$k$-modules}
By a \emph{$k$-module} we mean a differential graded $k$-module, in other words a $\Int$-graded chain complex of $k$-vector spaces. Following Dwyer~\cite{DwyerNoncommutative}, a $k$-module concentrated in degree 0 will be referred to as a \emph{discrete} $k$-module. A map of $k$-modules is just a chain map.

The category of $k$-modules has a cofibrantly generated model category structure described by Hovey in \cite[2.3]{HoveyBook}. In this model category structure the weak equivalences are quasi isomorphisms and the fibrations are degreewise surjections. In addition, this is a symmetric monoidal model category \cite[Proposition 4.2.13]{HoveyBook}, with the monoidal structure being the usual tensor product $\otimes$ of chain complexes over $k$. A monoid with respect to $\otimes$ is simply a dga.

Note that we define the suspension of a $k$-module $X$ to be $\Sigma X = (\Sigma k) \otimes X$.

\subsection{The category of $A$-modules}
Fix a dga $A$ over a field $k$. Since $A$ is a monoid with respect to $\otimes$ one can define the category of left $A$-modules as in \cite{SchwedeShipleyAlgebrasAndModules}. We shall describe these explicitly. An \emph{$A$-module} a differential graded \emph{left} $A$-module. A morphism of $A$-modules $f:M \to N$ is a morphism of chain complexes of degree zero which commutes with the action of $A$. The resulting category of $A$-modules is clearly an abelian category.

The category of $A$-modules has a Quillen model category structure where the weak equivalences are quasi-isomorphisms and fibrations are degreewise surjections, see \cite[Theorem 4.1]{SchwedeShipleyAlgebrasAndModules}. Hence every object is fibrant in this model category structure. The \emph{derived category of $A$-modules}, denoted $\Derived(A)$, is the homotopy category of this Quillen model category. As is well known, $\Derived(A)$ is a triangulated category and a short exact sequence of $A$-modules induces an exact triangle in $\Derived(A)$, since it is a homotopy fibration sequence. We define $\ext^n_A(X,Y)$ to be $\hom_{\Derived(A)}(\Sigma^{-n} X,Y)$.

Since $\otimes$ is symmetric, given a dga $A$ we can define its opposite $A^\op$, and $A^e=A\otimes A^\op$ is also a dga. An $A$-bimodule is simply an $A^e$-module, and thus the derived category of $A$-bimodules is $\Derived(A^e)$. When considering $A$ as an $A^e$-module we always mean the obvious bimodule structure on $A$ (there can be other bimodule structures, e.g. $A \otimes k \cong A$,  but they will play no part in this paper).

The category of $A$-bimodules has a tensor product $-\otimes_A-$. This tensor product is not symmetric and its unit $A$ is not cofibrant. Nevertheless there is a monoidal structure $-\otimes_A^\mathbf{L}-$ on the derived category of bimodules.

\subsection{Cones and cylinders}
\label{sub: Cones and cylinders}
The \emph{cone} of a map $X \to Y$ of $A$-modules is the $A$-module $\cone f=(\Sigma X^\natural \oplus Y^\natural, d)$ where $d=(-d^X,f+d^Y)$. The map $Y \to \coker f$ factors through a natural map $c_f:Y \to \cone f$ whose cokernel is $\Sigma X$. We write $\cone X$ for $\cone 1_X$ and $c_X$ for $c_{1_X}$. In fact, $\cone X = (\cone k) \otimes X$ and the map $\cone X \to \Sigma X$ is simply $(\cone k \to \Sigma k) \otimes X$.

From~\cite{SchwedeShipleyAlgebrasAndModules} and \cite[2.3]{HoveyBook} we learn that the generating cofibrations for $A$-modules are $\mathcal{I}=\{\Sigma^n A \to \cone\Sigma^n A \}$ and the generating acyclic cofibrations are $\mathcal{J}=\{0 \to \cone \Sigma^n A\}$. From this one easily sees that for any cofibrant $A$-module $X$, the map $X \to \cone X$ is a cofibration. Since $\cone f$ is the pushout of $\cone X \leftarrow X \xrightarrow{f} Y$ we see that $Y \to \cone f$ is a cofibration whenever $X$ is cofibrant.

\begin{remark}
It is easy to see that an $A$-module $X$ is in $\mathcal{I}$-cell \cite[Definition 2.1.9]{HoveyBook} if and only if $X$ is a semi-free $A$-module \cite{FelixHalperinThomasDGA}. Hence whenever $X$ is cofibrant then $X$ is a retract of a semi-free $A$-module.
\end{remark}

Given an $A$-module $X$ we define the \emph{cylinder} of $X$ to be $X \wedge I = \cone(X \xrightarrow{(1,1)} X \oplus X)$. There are obvious maps $X\oplus X \to X\wedge I \to X$. It is a simple exercise to show that whenever $X$ is cofibrant then $X \wedge I$ is a very good cylinder object for $X$ \cite[Definition 4.2]{DwyerSpalinski}, i.e. $X\oplus X \to X\wedge I$ is a cofibration and $X\wedge I \to X$ is an acyclic fibration.

\section{Differential graded preliminaries}
\label{sec: Differential graded preliminaries}

\subsection{Realization of simplicial $A$-modules}
The construction we name realization is simply the homotopy colimit of a simplicial $A$-module. There are other well known choices for this homotopy colimit, we have chosen one whose good properties are easy to prove.

\begin{definition}
Let $X_{(\bullet)}$ be a simplicial $A$-module, thus each $X_{(n)}$ is an $A$-module with the face maps $d_i:X_{(n)} \to X_{(n-1)}$, $0\leq i\leq n$ being $A$-module morphisms (we will only concern ourselves with the face maps). Define the \emph{realization} of $X_{(\bullet)}$ to be the $A$-module $|X_{(\bullet)}|$ whose underlying graded $A^\natural$-module is
\[ \bigoplus_p (\Sigma^p X_{(p)})^\natural\]
with differential $\partial$ given by
\[ \partial|_{\Sigma^p X_{(p)}^\natural} = (-1)^p\partial^{X_{(p)}} + \sum_{i=0}^p (-1)^i d_i .\]
It is not difficult to see that $|X_{(\bullet)}|$ is an $A$-module. Thus, realization is a functor from simplicial $A$-modules to $A$-modules.
\end{definition}

This realization is none other than the total complex of the double complex generated from $X_{(\bullet)}$, where the additional differential is simply $\sum_{i=0}^p (-1)^i d_i$.

\begin{definition}
Given an $A$-module $M$ the \emph{constant simplicial $A$-module $M_{(\bullet)}$} is given by $M_{(n)}=M$ for all $n$ and all maps are the identity map. It is easy to see that the realization of the constant simplicial $A$-module $M_{(\bullet)}$ has a natural equivalence $|M_{(\bullet)}| \xrightarrow{\simeq} M$.
\end{definition}

\subsection{The bar construction}
The (unnormalized two-sided) bar construction provides a model for the derived tensor product of a right $A$-module with a left $A$-module. We recall this construction next.

\begin{definition}
Let $M$ be a right $A$-module and let $N$ be a left $A$-module. The simplicial $A$-module $B_{(\bullet)}(M,A,N)$ is
\[ B_{(t)} (M,A,N) = M \otimes_k A^{\otimes t} \otimes_k N \]
where $A^{\otimes t}=\underbrace{A\otimes_k \cdots \otimes_k A}_{t\text{ times}}$. The face maps are
\[ d_i=
 \begin{cases}
  \eta_M \otimes 1_A^{t-1} \otimes 1_N & i=0 \\
  1_M \otimes 1_A^{i-1} \otimes \mu \otimes 1_A^{t-i-1} \otimes 1_N & 0<i<t\\
  1_M \otimes 1_A^{t-1} \otimes \eta_N & i=t
 \end{cases}
\]
where $\mu:A\otimes A \to A$ is the multiplication map and $\eta_M$ and $\eta_N$ are the module structure maps $\eta_M:M \otimes A \to A$ and $\eta_N:A\otimes N \to N$. We have no use for the degeneracy maps and so we forgo their description. The \emph{bar construction} is the realization $|B_{(\bullet)}(M,A,N)|$ which we shall denote by $M \ba N$.
\end{definition}

When $M$ and $N$ are $A$-bimodules then $M\ba N$ is again an $A$-bimodule, with the left $A$-action coming from the left $A$-action on $M$ and the right $A$-action coming from the right $A$-action on $N$. The following is well known.

\begin{lemma}
\label{lem: Map from derived A-tensor to non derived}
Let $M$ be a right $A$-module and let $N$ be a left $A$-module, then $M \ba N$ is a model for the derived tensor product $M \otimes_A^\mathbf{L} N$ and there exists a natural map $M \ba N \to M \otimes_A N$.
\end{lemma}
\begin{proof}
We shall only specify the natural map mentioned in the lemma. Note that $M \otimes_A N$ is the coequalizer of $d_0,d_1:M \otimes A \otimes N \rightrightarrows M \otimes N$. It is now obvious how to define a natural map of simplicial modules
\[ B_{(\bullet)}(M,A,N) \to (M\otimes_A N)_{(\bullet)} \, .\]
The natural map we are after is then the composition
\[ |B_{(\bullet)}(M,A,N)| \to |(M\otimes_A N)_{(\bullet)}| \xrightarrow{\sim} M \otimes_A N .\]
\end{proof}

\begin{remark}
A generic element of $M \ba N$ is denoted by $m[a]n$, where $m\in M$, $[a]=[a_1|\cdots|a_p] \in A^{\otimes p}$ and $n\in N$. The natural map $M \ba N \to M \otimes_A N$ is then
\[ m[a]n \ \mapsto  \
\begin{cases}
  m\otimes n & p=0, \\
  0 & \text{otherwise.}
 \end{cases}
\]
\end{remark}

\subsection{Properties of the bar construction}
We show that the bar construction is associative and has maps which can play the role of unit maps, although they are not isomorphisms.

\begin{definition}
Let $L$, $M$ and $N$ be $A^e$-modules. There is a natural isomorphism of $A^e$-modules $\alpha: (L \ba M) \ba N \to L \ba (M \ba N)$ which we now describe. A generic element of $(L \ba M) \ba N$ is denoted by $(l [a] m) [b] n$ where $l \in L$, $m \in M$, $n \in N$ and $[a]=[a_1|\cdots|a_p] \in A^{\otimes p}$ and $[b]=[b_1|\cdots|b_q] \in A^{\otimes q}$. Note that the degree of $[a]$ is $\deg a_1+\cdots +\deg a_p$ while the degree of $l [a] m$ is $\deg l + \deg [a] + \deg m +p$ and the degree of $(l [a] m) [b] n$ is $\deg l + \deg [a] + \deg m +p + \deg [b] + \deg n + q$. The isomorphism $\alpha$ is given by
\[ (l [a] m) [b] n \ \mapsto \ (-1)^{(\deg l + \deg [a] +p)q} \ l[a](m[b]n) . \]
Roughly speaking, the sign comes from interchanging the order in which we realize the bisimplicial $A^e$-module $B_{(\bullet)}(B_{(\bullet)}(L,A,M),A,N)$.
\end{definition}

Proof of the following lemma is a simple calculation of signs and degrees and is therefore omitted.
\begin{lemma}
The bar construction $\ba$ together with the associativity isomorphism $\alpha$  satisfy the associativity diagram (5) in \cite[VII.1]{MacLaneCategoriesBook}.
\end{lemma}

\begin{definition}
Define the left unit map $el_M:A \ba M \to M$ to be the composition $A \ba M \xrightarrow{\sim} A \otimes_A M \xrightarrow{\cong} M$. Clearly this map is an equivalence. We define the right unit map $er_M:M \ba A \to A$ similarly. It is easy to see that $el_A=er_A:A\ba A \to A$. We will usually denote both units by $e$. We caution the reader that these unit maps do not satisfy diagram (7) from \cite[VII.1]{MacLaneCategoriesBook}.
\end{definition}

The following properties are clear.
\begin{lemma}
The bar construction preserves equivalences and short exact sequences in either variable. The bar construction is bilinear in the following sense: let $f,g:M \to N$ and $h:X \to Y$ be maps of $A$-bimodules then $(f+g)\ba h = (f\ba h) +  (g \ba h)$ and $h\ba (f+g) = (h\ba f) + (h\ba g)$.
\end{lemma}

\subsection{Signs}
\label{sub: Signs}
For $k$-modules $X$ and $Y$ there are natural isomorphisms $\Sigma(X\otimes Y) \cong (\Sigma X) \otimes Y$ and $\Sigma(X\otimes Y) \cong X \otimes (\Sigma Y)$. The first isomorphism is simply $\Sigma(x\otimes y)\mapsto (\Sigma x)\otimes y$ while the second isomorphism is $\Sigma(x\otimes y)\mapsto (-1)^{|x|}x\otimes (\Sigma y)$. These standard isomorphisms induce in an obvious way isomorphisms of $A^e$-modules $\Sigma(X \ba Y) \cong (\Sigma X) \ba Y$ and $\Sigma(X \ba Y) \cong X \ba (\Sigma Y)$. From now on these are the isomorphisms we shall use (at times implicitly) whenever we consider a map of $A^e$-modules $X \otimes (\Sigma Y) \to M$ as an element of $\ext^0_{A^e}(\Sigma (X \ba Y),M) = \ext^{-1}_{A^e}(X \ba Y,M)$.

Similarly there are standard natural isomorphisms $(\cone X) \otimes Y \cong \cone(X \otimes Y) \cong X \otimes (\cone Y)$ which give a commutative diagram
\[ \xymatrixcompile{
{\cone(X) \otimes Y} \ar@{<->}[r]^\cong \ar[d]& {\cone(X \otimes Y)} \ar@{<->}[r]^\cong \ar[d] & { X \otimes (\cone Y)} \ar[d]\\
{\Sigma(X) \otimes Y} \ar@{<->}[r]^\cong & {\Sigma(X \otimes Y)} \ar@{<->}[r]^\cong & { X \otimes (\Sigma Y)} .
}\]

Now consider the pushout $Q=\cone X \coprod_X \cone X$. Clearly $Q$ is equivalent to $\Sigma X$. Indeed there are two natural equivalences $l:Q \to \cone X \coprod_X 0 = \Sigma X$ and $r:Q \to 0 \coprod_X \cone X = \Sigma X$ coming from maps of the appropriate pushout diagrams. It is also easy to see there is a natural map $\xi: \Sigma X \to Q$ such that $l\xi=-1$ while $r\xi=1$.

Now let $P=(\cone X) \otimes Y \coprod_{X \otimes Y} X \otimes (\cone Y)$, then similarly there are two natural equivalences $l:P \to (\Sigma X) \otimes Y$ and $r:P \to X \otimes (\Sigma Y)$ coming from maps of the appropriate pushout diagrams. Combining the map $\xi$ mentioned above with the various standard isomorphisms yields maps
\[ \Sigma(X \otimes Y) \to \cone (X \otimes Y) \coprod_{X \otimes Y} \cone (X \otimes Y) \xrightarrow{\cong} (\cone X) \otimes Y \coprod_{X \otimes Y} X \otimes (\cone Y) = P,\]
we denote this composition by $\zeta$. Clearly, $r\zeta$ is the standard isomorphism while $l\zeta$ is -1 times the standard isomorphism. Moreover, there is the following natural short exact sequence
\[ \xymatrixcompile{
{\Sigma (X\otimes Y)} \ar[d]^{\zeta} \ar[r] & {\cone \Sigma (X\otimes Y)} \ar[d] \ar[r] & {\Sigma^2 (X\otimes Y)} \ar[d]^\cong \\
{P} \ar[r] & {\cone X \otimes \cone Y} \ar[r] & {\Sigma X \otimes \Sigma Y}
}\]
where the rightmost vertical isomorphism is the standard one.

These properties extend to the bar construction $\ba$ in an obvious manner, and so we give the following lemma without proof.
\begin{lemma}
\label{lem: Signs}
Let $X$ and $Y$ be $A^e$-modules. Then there exists a natural morphism of short exact sequences of $A^e$-modules
\[ \xymatrixcompile{
{\Sigma (X\ba Y)} \ar[d]^{\zeta} \ar[r] & {\cone \Sigma (X\ba Y)} \ar[d] \ar[r] & {\Sigma^2 (X\ba Y)} \ar[d]^\cong \\
{(\cone X) \ba Y \coprod_{X \ba Y} X \ba (\cone Y)} \ar[r] & {\cone X \ba \cone Y} \ar[r] & {\Sigma X \ba \Sigma Y}
}\]
where the rightmost vertical map is the standard isomorphism. In addition, the composition $\Sigma (X\ba Y) \xrightarrow{\zeta} (\cone X) \ba Y \coprod_{X \ba Y} X \ba (\cone Y) \to X \ba (\Sigma Y)$ is the standard isomorphism, while $\Sigma (X\ba Y) \xrightarrow{\zeta} (\cone X) \ba Y \coprod_{X \ba Y} X \ba (\cone Y) \to (\Sigma X) \ba Y$ is -1 times the standard isomorphism.
\end{lemma}

\begin{remark}
The signs here are analogous to the topological choice of orientation on the boundary of a manifold, compare \cite[Remark 2.2]{DuggerMultiplicative1}.
\end{remark}

\section{Construction of the spectral sequence}
\label{sec: Construction of the spectral sequence}

\subsection{Filtering the dga}
\label{sub: Filtering the dga}
Let $J(n)$ to be the sub-complex
\[0\to \cdots \to 0 \to A^n \to A^{n+1} \to A^{n+1} \to \cdots\]
of $A$. Clearly $J(n)$ is an $A$-bimodule and we have a filtration of $A$ by $A$-bimodules
\begin{equation}
\label{equ: filtration of A}
\cdots \to J(n) \xrightarrow{\iota} J(n-1) \xrightarrow{\iota} \cdots \to J(0)=A .
\end{equation}
There are also short exact sequences of $A$-bimodules
\begin{equation}
\label{equ: triangle for J(n)}
0 \to J(n+1) \xrightarrow{\iota} J(n) \xrightarrow{\theta} \Sigma^{-n} A^n \to 0 .
\end{equation}
The tower (\ref{equ: filtration of A}) induces morphisms
\begin{equation}
\cdots \to \ext_{A^e}^*(A,J(n)) \xrightarrow{\iota} \ext_{A^e}^*(A,J(n-1)) \to \cdots \to \ext_{A^e}^*(A,A)
\end{equation}
and the short exact sequences (\ref{equ: triangle for J(n)}) induce long exact sequences sequences:
\begin{multline}
\cdots \to \ext^t_{A^e}(A,J(n+1)) \xrightarrow{\iota} \ext^t_{A^e}(A,J(n)) \xrightarrow{\theta} \ext^t_{A^e}(A,\Sigma^{-n} A^n)\\ \xrightarrow{\kappa} \ext^{t+1}_{A^e}(A,J(n+1)) \to \cdots
\end{multline}
The indeterminacy of the connecting homomorphism can cause us trouble when trying to work out signs. Therefore we define $\kappa:\Sigma^{-n} A^n \to \Sigma J(n+1)$ to be the morphism in $\Derived(A^e)$ represented by the obvious maps $\Sigma^{-n} A^n \to \cone \theta \xleftarrow{\sim} \Sigma J(n+1)$.

\subsection{The spectral sequence}
The spectral sequence we build shall be homologically graded, hence we set:
\begin{align*}
D^1_{p,q}&= \ext^{-p-q}_{A^e}(A,J(-p))  
\\
E^1_{p,q}&= \ext^{-p-q}_{A^e}(A,\Sigma^{p}A^{-p}) = \ext^{-q}_{A^e}(A,A^{-p}) . 
\end{align*}
The morphisms $\kappa$,  $\iota$ and $\theta$ now become
\begin{align*}
E^1_{p,q} &\xrightarrow{\kappa} D^1_{p-1,q}\\
D^1_{p,q} &\xrightarrow{\iota} D^1_{p+1,q-1}\\
D^1_{p,q} &\xrightarrow{\theta} E^1_{p,q} .
\end{align*}
Together these yield an exact couple which gives rise to the desired spectral sequence $(E^r,d^r)$.

\begin{lemma}
\label{lem: Conditional convergence}
Let $A$ be a coconnective augmented $k$-dga. Then the spectral sequence constructed above conditionally converges (see~\cite[Definition 5.10]{BoardmanCCSpecSeq}) to $\ext^{-p-q}_{A^e}(A,A)$.
\end{lemma}
\begin{proof}
To prove the lemma we need to extend our filtration of $A$ by setting $J(-p)=A$ for all $p>0$. Now we have a similar spectral sequence, only $D^1_{p,q} = \ext^{-p-q}_{A^e}(A,A)$ for $p>0$ (this leaves the $E^1$-term unchanged).

Fix an index $q$ and consider the tower of graded abelian groups $M(p)=H_q(J(p))$. This tower satisfies the trivial Mittag-Leffler condition and therefore $\lim_p M(p) = \lim^1 M(p) = 0$. From this we conclude that the homotopy limit $\holim_p J(p)$ is equivalent to zero. Since
\[\holim_p \mathbf{R}\Hom_{A^e}(A,J(p)) \simeq \mathbf{R}\Hom_{A^e}(A,\holim_p J(p)) \simeq 0 , \]
then by a Milnor type short exact sequence (see for example \cite[Theorem 4.9]{BoardmanCCSpecSeq}) we see that $\lim_{p\to-\infty} D^1_{p,*} = \lim^1_{p\to-\infty} D^1_{p,*} = 0$.
Thus our spectral sequence converges conditionally to $\colim_{p\to \infty} D^1_{p,*}=\ext^{-p-q}_{A^e}(A,A)$.
\end{proof}

\section{The multiplicative structure}
\label{sec: The multiplicative structure}

\subsection{The pairing on the filtration}
The multiplication map $A\otimes_k A \to A$ yields an associative pairing
\[ J(n) \otimes_k J(m) \xrightarrow{\phi_{n,m}} J(n+m).\]
Note that the maps $\phi_{n,m}$ are maps of $A$-bimodules, where the bimodule structure on $J(n) \otimes_k J(m)$ comes from the left $A$-module structure on $J(n)$ and the right $A$-module structure on $J(m)$. This is not sufficient, however, since we need to construct a pairing of $A$-bimodules:
\[ J(n) \ba J(m) \xrightarrow{\psi_{n,m}} J(n+m) . \]

Let $J(n+m)_{(\bullet)}$ be the constant simplicial bimodule $J(n+m)$. Consider the map of simplicial of $A$-bimodules
\[\psi_{(\bullet),n,m}:B_{(\bullet)}(J(n),A,J(m)) \to J(n+m)_{(\bullet)}\]
which is given by the obvious multiplication map
\[ \psi_{(t),n,m}: J(n) \otimes_k A^{\otimes t} \otimes_k J(m) \to J(n+m) . \]
This results in a map of simplicial $A$-bimodules because the morphisms $\psi_{(t),m,n}$ commute with the face (and degeneracy) maps of $B_{(\bullet)}(J(n),A,J(m))$.

Upon taking realization of both simplicial bimodules we get a map $|\psi_{(\bullet),n,m}|:J(n) \ba J(M) \to |J(n+m)_{(\bullet)}|$. Composing this map with the natural weak equivalence $|J(n+m)_{(\bullet)}|\xrightarrow{\simeq} J(n+m)$ yields
\[ \psi_{n,m}: J(n) \ba J(M) \to J(n+m),\]
which is the pairing we need. We will omit the subscripts $n$ and $m$ whenever they are clear from the context. Note that $\psi_{0,0}$ is the unit map $e$.

\begin{lemma}
\label{lem: Pairing psi commutes with iota}
\[ \psi(\iota\ba 1)=\iota \psi = \psi(1\ba \iota) . \]
\end{lemma}
\begin{proof}
We will only show $\psi(\iota\ba 1)=\iota \psi$, the proof of the second equality being similar. Clearly, the diagram below commutes
\[ \xymatrixcompile{
{B_{(p)}\big(J(s),A,J(t)\big)} \ar[rr]^{\iota \otimes 1} \ar[d] & &
 {B_{(p)}\big(J(s-1),A,J(t)\big)} \ar[d] \\
{J(s+t)} \ar[rr]^\iota & & {J(s+t-1)},
}\]
which implies the following diagram of $A$-modules commutes
\[ \xymatrixcompile{
{J(s) \ba J(t)} \ar[rr]^{\iota \ba 1} \ar[d]^{\psi_{s,t}} & &
 {J(s-1) \ba J(t)} \ar[d]^{\psi_{s-1,t}} \\
{|J(s+t)_{(\bullet)}|} \ar[rr]^{|\iota|} \ar[d]^\simeq & & {|J(s+t-1)_{(\bullet)}|} \ar[d]^\simeq \\
{J(s+t)} \ar[rr]^\iota & & {J(s+t-1)}.
}\]
\end{proof}

\subsection{The pairing on $\mathbf{D^1}$}
\label{sub: Pairing on D1}
We now construct a pairing on the $D^1$ term of the exact couple. Define maps
\[ \bar{\psi}_{n,m}: \ext^*_{A^e}(A, J(n)) \otimes \ext^*_{A^e}(A, J(m)) \to \ext^*_{A^e}(A, J(n+m))\]
in the following way: given $f\in \ext^{-s}_{A^e}(A,J(n))$ and $g\in\ext^{-t}_{A^e}(A,J(m))$ let $\bar{\psi}_{n,m}(f\otimes g)$ be the composition
\[ \Sigma^{s+t}A \xrightarrow{\cong} \Sigma^s A\ba \Sigma^t A \xrightarrow{f\ba g} J(n) \ba J(m) \xrightarrow{\psi_{n,m}} J(n+m)\]
in $\Derived(A^e)$, where the leftmost isomorphism is $e^{-1}$. Bilinearity of the bar construction $\ba$ allows us to extend this to a pairing on $\ext^*_{A^e}(A, J(n)) \otimes_k \ext^*_{A^e}(A, J(m))$.

For what follows we shall need a concrete representation of this pairing on the category of $A^e$-modules. First, we must fix a cofibrant replacement of $A$ as an $A^e$-module. Since $A\ba A$ is cofibrant we define $e:A\ba A \to A$ to be our cofibrant replacement. We will usually denote $A\ba A$ by $\tilde{A}$. An element $f\in \ext^{-s}_{A^e}(A,J(n))$ is now represented by the following maps of $A^e$-modules: $\Sigma^{s}A \xleftarrow{e} \Sigma^{s}\tilde{A} \xrightarrow{\tilde{f}} J(n)$, where $\tilde{f}$ represents $fe^{-1}$. We shall usually omit $e$ from the description, simply saying $\tilde{f}$ represents $f$.

Given another element $g\in\ext^{-t}_{A^e}(A,J(m))$ we choose a map of $A^e$-modules $\tilde{g}:\Sigma^{t} \tilde{A} \to J(m)$ representing $g$. Now we have the following maps of $A^e$-modules
\[ \xymatrixcompile{
{\Sigma^{s+t} A} & {\Sigma^{s} A \ba \Sigma^{t} A} \ar[l]^\sim_{e} & {\Sigma^{s} \tilde{A}\ba \Sigma^{t} \tilde{A}} \ar[l]_{e \ba e}^\sim \ar[r]^-{\tilde{f} \ba \tilde{g}}  & {J(n) \ba J(m)} \ar[r]^{\psi_{n,m}} & {J(n+m)}. }\]
It is easy to see that $\bar{\psi}_{n,m}(f\otimes g)$ is equal to the composition $\psi_{n,m} (\tilde{f} \ba \tilde{g})(e\ba e)^{-1} e^{-1}$ in $\Derived(A^e)$.

>From now on when we say that a map $f:\tilde{A} \ba \tilde{A} \to X$ represents an element of $\ext_{A^e}^*(A,X)$ we implicitly refer to the diagram $A \xleftarrow{e} \tilde{A} \xleftarrow{e\ba e} \tilde{A} \ba \tilde{A} \xrightarrow{f} X$. Thus, we have just shown that $\psi_{n,m} (\tilde{f} \ba \tilde{g})$ represents $\bar{\psi}_{n,m}(f\otimes g)$.

\begin{lemma}
\label{lem: Associativity of pairing psi}
The pairing $\psi$ is associative and so is the induced pairing $\bar{\psi}$.
\end{lemma}
\begin{proof}
To show that $\psi$ is associative we must show the following diagram commutes:
\[ \xymatrixcompile{
{J(l) \ba \big(J(m) \ba J(n)\big)} \ar[rrr]^{J(l) \ba \psi_{m,n}} \ar[d]^\cong_\alpha & & &
  {J(l) \ba J(m+n)} \ar[dd]^{\psi_{l,m+n}} \\
{\big(J(l) \ba J(m)\big) \ba J(n)} \ar[d]^{\psi_{l,m} \ba J(n)} \\
{J(l+m) \ba J(n)} \ar[rrr]^{\psi_{l+m,n}} & & & {J(l+m+n)}.
}\]
Let $[a]=[a_1|\cdots|a_p] \in A^{\otimes p}$, $[b]=[b_1|\cdots|b_p] \in A^{\otimes q}$ and let $j_1[a](j_2[b]j_3)$ be a generic element in $J(l) \ba (J(m) \ba J(n))$. A simple calculation shows that
\begin{align*}
\psi(1\ba \psi) (j_1[a](j_2[b]j_3)) &=
\begin{cases}
  j_1 j_2 j_3 & p=q=0 \\
  0 & \text{otherwise}
 \end{cases}
\\ &= \psi(\psi \ba 1) \alpha (j_1[a](j_2[b]j_3))
\end{align*}

To show that $\bar{\psi}$ is also associative, we must show that our choice of isomorphism $A \cong A \ba A$ in $\Derived(A^e)$ is coassociative. This reduces to showing that $A \ba A \xrightarrow{e} A$ is associative. Since $e=\psi_{0,0}$, it is indeed associative by the first part of the proof.
\end{proof}

\begin{lemma}
The pairing $\bar{\psi}_{0,0}$ is the standard multiplication on $HH^*(A)$.
\end{lemma}
\begin{proof}
This is well known and so we shall only sketch the proof, which uses the Eckmann-Hilton argument. The standard multiplication on $\ext_{A^e}^*(A,A)$ is done by composition of arrows, denoted $f\circ g$. It is easy to show that the identity morphism $1_A$ is a unit also for $\bar{\psi}$. Thus, for any $f\in HH^*(A)$
\[ f \circ 1 = 1 \circ f = f = \bar{\psi}(1\otimes f) = \bar{\psi}(f\otimes 1) . \]
It is also a simple exercise to show that for any $f_1,f_2,g_1,g_2 \in HH^*(A)$
\[ \bar{\psi}((f_1\circ f_2) \otimes (g_1 \circ g_2)) = \bar{\psi}(f_1 \otimes g_1) \circ \bar{\psi}(f_2 \otimes g_2). \]
Thus, by the Eckmann-Hilton argument both multiplications are the same and are associative and graded-commutative.
\end{proof}

\subsection{The multiplication on $\mathbf{E^1}$}
The multiplication on the $E^1$ term arises from the pairing $\psi$ in a standard way, which we will now follow (compare \cite{DuggerMultiplicative1}). For convenience we shall denote the $A$-bimodule $\Sigma^{-n}A^n$ by $\A^n$.

\begin{lemma}
There is a short exact sequence of $A$-bimodules
\[ 0 \to \binom{J(n) \ba J(m+1) +}{ J(n+1) \ba J(m)} \to J(n) \ba J(m) \to \A^n \ba \A^m \to 0\]
where $J(n) \ba J(m+1)+J(n+1) \ba J(m)$ is the sum as subcomplexes of $J(n) \ba J(m)$.
\end{lemma}
\begin{proof}
There are short exact sequences:
\[0 \to \binom{B_{(t)}(J(n),A,J(m+1))+}{B_{(t)}(J(n+1),A,J(m))} \to B_{(t)}(J(n),A,J(m)) \to B_{(t)}(\A^n,A,\A^m) \to 0\]
which yield the desired short exact sequence after taking realization.
\end{proof}

It is a simple observation that $J(n) \ba J(m+1)+J(n+1) \ba J(m)$ is in fact the pushout
\[J(n+1)\ba J(m) \!\!\!\!\!\!\!\!\!\! \coprod_{J(n+1)\ba J(m+1)} \!\!\!\!\!\!\!\!\!\! J(n)\ba J(m+1) .\]
The following lemma is an immediate consequence of this observation and the fact that $\psi$ commutes with $\iota$.

\begin{lemma}
The following diagram of $A$-bimodules commutes
\[ \xymatrixcompile{
{\displaystyle \binom{J(n) \ba J(m+1)+}{J(n+1) \ba J(m)}} \ar[d]^{\psi'_{n,m}} \ar[r] & {J(n) \ba J(m)} \ar[d]^{\psi_{n,m}} \\
{J(n+m+1)} \ar[r]^\iota& {J(n+m)} } \]
where $\psi'$ is the obvious map of subcomplexes.
\end{lemma}

\begin{definition}
\label{def: The pairing mu}
Define a pairing $\mu_{n,m}:\A^n \ba \A^m \to \A^{n+m}$ of $A$-bimodules to be the pairing induced from the following diagram where both rows are short exact sequences of complexes:
\[ \xymatrixcompile{
{\displaystyle \binom{J(n) \ba J(m+1)+}{J(n+1) \ba J(m)}} \ar[d]^{\psi'_{n,m}} \ar[r] & {J(n) \ba J(m)} \ar[r] \ar[d]^{\psi_{n,m}} & {\A^n \ba \A^m} \ar[d]^{\mu_{n,m}} \\
{J(n+m+1)} \ar[r]& {J(n+m)} \ar[r] & {\A^{n+m}} . } \]
The multiplication on the $E^1$ term follows easily from the pairing above. Define
\[\bar{\mu}_{n,m}:\ext^*_{A^e}(A,\A^n)\otimes_k \ext^*_{A^e}(A, \A^m) \to \ext^*_{A^e}(A, \A^{n+m})\]
as the composition $\mu_{n,m}(-\ba-)e^{-1}$.
\end{definition}

\begin{lemma}
\label{lem: Associativity of pairing mu}
The pairing $\mu$ is associative and so is the multiplication $\bar{\mu}$.
\end{lemma}
\begin{proof}
That $\mu$ is associative can be deduced from Lemma~\ref{lem: Associativity of pairing psi}, but a direct calculation is much easier and makes the associativity obvious. As is the proof of Lemma~\ref{lem: Associativity of pairing psi}, associativity of $\bar{\mu}$ is a consequence of the associativity of $\mu$ and of $e$.
\end{proof}

\subsection{Representing multiplication on $\mathbf{E^1}$}
We first show how to represent certain elements in $E^1$. The following lemma is well known in many settings, see for example \cite[Lemma 3.3]{DuggerMultiplicative1}.

\begin{lemma}
\label{lem: Representing elements in E1}
Let $x\in \ext_{A^e}^s(A,\A^p)$ and $a \in \ext_{A^e}^{s+1}(A,J(p+n+1))$ such that $\kappa x = \iota^n a$. Then there exists a commutative diagram of short exact sequences of $A^e$-modules
\[ \xymatrixcompile{
{\Sigma^{-s-1}\tilde{A}} \ar[r] \ar[d]^w & {\cone \Sigma^{-s-1} \tilde{A}} \ar[d]^f \ar[r] & {\Sigma^{-s}\tilde{A}} \ar[d]^{\tilde{x}}\\
{J(p+1)} \ar[r]^\iota & {J(p)} \ar[r]^\theta & {\A^p} } \]
and a map of $A^e$-modules $\tilde{a}: \Sigma^{-s-1}\tilde{A} \to J(p+n+1)$ such that
\begin{enumerate}
\item $\tilde{a}$ represents $a$,
\item $w=\iota^n \tilde{a}$ and so $w$ represents $\kappa x$,
\item $\tilde{x}$ represents $x$.
\end{enumerate}
\end{lemma}
\begin{proof}

The case for $n=0$ is well known and therefore omitted. As in the proof of \cite[Lemma 3.3]{DuggerMultiplicative1}, the case for $n>0$ is done using the homotopy extension property and we shall only sketch the argument. Recall that for an $A^e$-module $M$ we denote by $M\wedge I$ the mapping cone of $M \xrightarrow{(1,-1)} M \oplus M$. Let the two obvious maps $M \to M \wedge I$ be denoted by $i_o$ and $i_1$. Since $\tilde{A}$ and $\cone \tilde{A}$ are cofibrant, $\tilde{A}\wedge I$ and $\cone \tilde{A} \wedge I$ are very good cylinder objects for $\tilde{A}$ and $\cone \tilde{A}$ respectively (see~\ref{sub: Cones and cylinders}).

So suppose we have found $w:\Sigma^{-s-1}\tilde{A} \to J(p+1)$ and $f: \cone \Sigma^{-s-1} \tilde{A} \to J(p)$ as for the case $n=0$. Choose a map $\tilde{a}: \Sigma^{-s-1}\tilde{A} \to J(p+n+1)$ which represents $a$. Since $\iota \tilde{a}$ and $w$ need not be equal, we need to replace $f$ by an equivalent map $f'$ which will make the following diagram commute
\[ \xymatrixcompile{
{\Sigma^{-s-1}\tilde{A}} \ar[r] \ar[d]^{\iota \tilde{a}} & {\cone \Sigma^{-s-1} \tilde{A}} \ar[d]^{f'} \\
{J(p+1)} \ar[r]^\iota & {J(p)} .  } \]

Equivalence of $\iota \tilde{a}$ and $w$ implies there is a map $h:\Sigma^{-s-1}\tilde{A} \wedge I\to J(p+1)$ which is a homotopy between these two maps, thus $hi_1=\iota a$ and $hi_0=w$. Let $Y$ be the pushout of the diagram
\[\Sigma^{-s-1}\tilde{A} \wedge I \xleftarrow{i_0} \Sigma^{-s-1}\tilde{A} \xrightarrow{ } \cone \Sigma^{-s-1} \tilde{A}\]
(one can liken $Y$ to $\tilde{A}\times [0,1] \cup_{\tilde{A}\times\{0\}} \cone \tilde{A} \times \{0\}$). Clearly there is a natural map $h':Y \to J(p)$. Since $Y \to \cone \Sigma^{-s-1} \tilde{A} \wedge I$ is a cofibration, we can extend $h'$ to a homotopy $h'':\cone \Sigma^{-s-1} \tilde{A} \wedge I \to J(p)$. Now $h''i_1:\cone \Sigma^{-s-1} \tilde{A} \to J(p)$ is the map $f'$ we need.
\end{proof}

\begin{lemma}
\label{lem: Representing multiplication on E1}
Suppose $x\in \ext_{A^e}^{-s-1}(A,\A^p)$ and $y\in \ext_{A^e}^{-t-1}(A,\A^u)$ are represented by diagrams
\[ \xymatrixcompile{
{\Sigma^{s}\tilde{A}} \ar[r] \ar[d]^w & {\cone \Sigma^{s} \tilde{A}} \ar[d]^f \ar[r] & {\Sigma^{s+1}\tilde{A}} \ar[d]^{\tilde{x}} &  &
{\Sigma^{t}\tilde{A}} \ar[r] \ar[d]^z & {\cone \Sigma^{t} \tilde{A}} \ar[d]^g \ar[r] & {\Sigma^{t+1}\tilde{A}} \ar[d]^{\tilde{y}} \\
{J(p+1)} \ar[r]^\iota & {J(p)} \ar[r]^\theta & {\A^p} & &
{J(u+1)} \ar[r]^\iota & {J(u)} \ar[r]^\theta & {\A^u} .} \]
Then $\kappa(xy)$ is represented by
\[  \psi'_{p,u}(f \ba z \coprod_{w \ba z} w \ba g) . \]
\end{lemma}
\begin{proof}
The proof is simply given by the following commutative diagram
\[\xymatrixcompile{ {\displaystyle \cone\Sigma^{s}\tilde{A} \ba \Sigma^{t}\tilde{A}
\!\!\!\!\! \coprod_{\Sigma^{s}\tilde{A} \ba \Sigma^{t}\tilde{A}} \!\!\!\!\!
\Sigma^{s}\tilde{A} \ba \cone \Sigma^{t}\tilde{A} } \ar[r] \ar[d]^-{f \ba z \coprod_{w \ba z} w \ba g}  &
{\cone \Sigma^{s}\tilde{A} \ba \cone \Sigma^{t}\tilde{A} } \ar[r] \ar[d]^{f \ba g} &
{\Sigma^{s+1}\tilde{A} \ba \Sigma^{t+1}\tilde{A}} \ar[d]^{\tilde{x} \ba \tilde{y}}
\\ {\displaystyle J(p) \ba J(u+1) \!\!\!\!\!  \coprod_{J(p+1)\ba J(u+1)} \!\!\!\!\!  J(p+1) \ba J(u)}  \ar[r] \ar[d]^-{\psi'} & {J(p) \ba J(u)} \ar[d]^\psi \ar[r] & {\A^p \ba \A^u} \ar[d]^\mu\\
{J(p+u+1)} \ar[r] & {J(p+u)} \ar[r] & {\A^{p+u}}  . }\]
Note that we are implicitly using here the natural isomorphisms described in Lemma~\ref{lem: Signs}.
\end{proof}

\subsection{The multiplicative property of the spectral sequence}
We shall employ a criterion of Massey from~\cite{Massey} to show that the spectral sequence is multiplicative. Note that translating the proof of~\cite[Proposition 5.1]{DuggerMultiplicative1} to our setting would work equally well. Translated to our setting, Massey's criterion is:

\begin{theorem}[Massey~\cite{Massey}]
Suppose the following conditions hold for the spectral sequence in Section~\ref{sec: Construction of the spectral sequence}.
\begin{enumerate}
\item $E^1$ is a graded algebra.
\item For every $x,y \in E^1$ and $a,b\in D^1$ such that $\kappa(x)=\iota^n(a)$ and $\kappa(y)=\iota^n(b)$ there exists $c\in D^1$ such that $\kappa(xy)=\iota^n(c)$ and $\theta(c)=\theta(a)y+(-1)^{|x|}x\theta(b)$.
\end{enumerate}
Then the spectral sequence is multiplicative.
\end{theorem}

\begin{proposition}
\label{pro: Multiplicativity of the HH spectral sequence}
The pairings $\mu$ and $\psi$ defined above make the spectral sequence $E^r_{*,*}$ into a multiplicative spectral sequence. Namely for every $r$, $E^r_{*,*}$ has an induced multiplication for which $d^r$ is a derivation.
\end{proposition}
\begin{proof}
In light of Lemma~\ref{lem: Associativity of pairing mu} the first condition of Massey's criterion reduces to verifying that our grading choice is correct, which is easily checked.

For the second condition we represent $x$ and $y$ by diagrams as in Lemma~\ref{lem: Representing elements in E1}:
\[ \xymatrixcompile{
{\Sigma^{s}\tilde{A}} \ar[r] \ar[d]^w & {\cone \Sigma^{s} \tilde{A}} \ar[d]^f \ar[r] & {\Sigma^{s+1}\tilde{A}} \ar[d]^{\tilde{x}} & &
{\Sigma^{t}\tilde{A}} \ar[r] \ar[d]^z & {\cone \Sigma^{t} \tilde{A}} \ar[d]^g \ar[r] & {\Sigma^{t+1}\tilde{A}} \ar[d]^{\tilde{y}} \\
{J(p+1)} \ar[r]^\iota & {J(p)} \ar[r]^\theta & {\A^p} & &
{J(u+1)} \ar[r]^\iota & {J(u)} \ar[r]^\theta & {\A^u} } \]
where $\tilde{x}$ represents $x$ and $\tilde{y}$ represents $y$. In addition, by the same lemma, there are maps $\tilde{a}: \Sigma^{s}\tilde{A} \to J(p+n+1)$ and $\tilde{b}: \Sigma^{t}\tilde{A} \to J(u+n+1)$ representing $a$ and $b$ such that $w=\iota^n \tilde{a}$ and $z=\iota^n b$.

Denote by $P$ the pushout
\[{ J(p) \ba J(u+n+1) \!\!\!\!\!\!\!\!\!\!\!\!\!\!\!\!\!  \coprod_{J(p+n+1)\ba J(u+n+1)} \!\!\!\!\!\!\!\!\!\!\!\!\!\!\!\!\!  J(p+n+1) \ba J(u)}\]
and by $B$ the pushout
\[ {\cone\Sigma^{s}\tilde{A} \ba \Sigma^{t}\tilde{A}
\!\!\!\!\! \coprod_{\Sigma^{s}\tilde{A} \ba \Sigma^{t}\tilde{A}} \!\!\!\!\!
\Sigma^{s}\tilde{A} \ba \cone \Sigma^{t}\tilde{A} }.\]
Let $\psi''$ be the map
\[ \psi_{p,u+n+1} \!\!\!\!\!\!\!\!\!\! \coprod_{\psi_{p+n+1,u+n+1}} \!\!\!\!\!\!\!\!\!\! \psi_{p+n+1,u}: P \to J(p+u+n+1) .\]
We can now define a map $\tilde{c}: B \to J(p+u+n+1)$ by
\[ \tilde{c}= \psi''(f \ba \tilde{b} \coprod_{\tilde{a} \ba \tilde{b}} \tilde{a} \ba g) . \]
We shall show that the morphism $c\in\ext_{A^e}^{-s-t-1}(A,J(p+u+n+1))$ which is represented by $\tilde{c}$ is the desired element in $D^1$. As before, we are implicitly using the isomorphism $B \cong \Sigma^{s+t+1} (\tilde{A} \ba \tilde{A})$ of Lemma~\ref{lem: Signs}.

It is easy to verify that the following diagram commutes
\[ \xymatrixcompile{
{B} \ar[d]_{f \ba \tilde{b} \coprod_{\tilde{a} \ba \tilde{b}} \tilde{a} \ba g} \ar[dr]^{(f \ba z \coprod_{w \ba z} w \ba g)} \\
{P} \ar[d]_{\psi''} \ar[r]^-{\iota^n} &
{\displaystyle J(p) \ba J(u+1) \!\!\!\!\!  \coprod_{J(p+1)\ba J(u+1)} \!\!\!\!\!  J(p+1) \ba J(u)} \ar[d]^{\psi'} \\
{J(p+u+n+1)} \ar[r]^{\iota^n}  & {J(p+u+1)} .
}\]
Hence, by Lemma~\ref{lem: Representing multiplication on E1}, we see that $\kappa(xy)=\iota^n c$.

It remains to show that $\theta(c)=\theta(a)y+(-1)^{|x|}x\theta(b)$. There is the following commutative diagram
\[ \xymatrixcompile{
{J(p) \ba J(u+n+1)} \ar[d] & {J(p+n+1) \ba J(u+n+1)} \ar[r]^-{1\otimes \iota^{n+1}} \ar[l]_-{\iota^{n+1} \otimes 1} \ar[d] & {J(p+n+1) \ba J(u)} \ar[d] \\
{\A^p \ba J(u+n+1)} & 0 \ar[r] \ar[l] & {J(p+n+1) \ba \A^u} .
} \]
Taking the pushout of each row yields a map $\tau:P \to \A^p \ba J(u+n+1) \oplus J(p+n+1) \ba \A^u$. We leave it to the reader to check that the following diagram commutes
\[ \xymatrixcompile{
{B} \ar[r]^-{f \ba \tilde{b} \coprod_{\tilde{a} \ba \tilde{b}} \tilde{a} \ba g} \ar[ddr]^{\tilde{c}} & {P} \ar[r]^-{\tau} \ar[dd]^{\psi''} &
 {\displaystyle \binom{\A^p \ba J(u+n+1)\oplus}{J(p+n+1) \ba \A^u}} \ar[d]^{1\ba \theta \oplus \theta \ba 1} \\
& &
 {\displaystyle \binom{\A^p \ba \A^{u+n+1}\oplus}{\A^{p+n+1} \ba \A^u}} \ar[d]^{\mu+\mu}\\
&
 {J(s+t+n+1)} \ar[r]^{\theta} & {A^{s+t+n+1}} .
} \]
>From this diagram one sees that $\theta \tilde{c}$ is equal to $\mu(\theta\tilde{a} \ba \tilde{y})+\mu(\tilde{x} \ba \theta\tilde{b})$. Using the sign convention from Lemma~\ref{lem: Signs} we see that $\theta (c) = \theta(a)y+(-1)^{|x|}x\theta(b)$.
\end{proof}

\section{The Hochschild cohomology shearing map}
\label{sec: The Hochschild cohomology shearing map}

In this section we give three equivalent definitions for the Hochschild cohomology shearing mapping
\[ \shr: HH^*(A) \to \ext^*_{A^e}(A,k) . \]
This morphism is well known and there is no claim to originality here. However, this is only one possible variant of the shearing map; a more thorough discussion can be found in \cite{BensonGreenleesShamir}.

\subsection{Two descriptions of the shearing map}
\label{sub: Two descriptions of the shearing map}
Let $A_l$ denote the dga $A \otimes k$ and let $\lambda:A^e \to A_l$ be the obvious map of dgas. Clearly $A_l$ is isomorphic to $A$. The map $\lambda$ induces an adjunction $F:A^e\modules \rightleftarrows A_l\modules:G$ where the left adjoint is $F(M)=M\otimes_A k$ and the right adjoint is $G(N) = N \cong N\otimes k$ with the left action of $A$ on $N$ and the right action of $A$ on $k$.

Since $G$ preserves fibrations and weak equivalences, $F$ and $G$ constitute a Quillen pair and thereby yield a pair of adjoint functors on the derived categories
\[ \F: \Derived(A^e) \rightleftarrows \Derived(A_l):\G . \]
It is easy to see that $\F(A)$ is isomorphic in $\Derived(A_l)$ to $k$.

\begin{definition}
\label{def: Sheering map}
Choose an isomorphism $\F(A) \xrightarrow{\varphi} k$ in $\Derived(A_l)$. Define the \emph{shearing map for $A$} to be the morphism $\shr:\ext^*_{A^e}(A,A) \to \ext^*_{A_l}(k,k)$ given by the composition
\[ \ext^*_{A^e}(A,A) \xrightarrow{\F} \ext^*_{A^l}(\F(A),\F(A)) \xrightarrow{p} \ext^*_{A_l}(k,k) , \]
where $p(f)=\varphi f \varphi^{-1}$. We note two immediate properties of the shearing map. First, it is clear that the shearing map is a map of graded rings. Second, when $A$ is simply connected the definition is independent of the choice of $\varphi$ because, in this case, the automorphism group of $k$ in $\Derived(A_l)$ is the commutative group $k^*$, which is also central in $\ext^*_{A_l}(k,k)$.
\end{definition}

Let $\eta_A:A \to \G\F(A)$ be the unit map of this adjunction. We can now define another map $\alpha:\ext^*_{A^e}(A,A) \to \ext^*_{A_l}(k,k)$, as the composition
\[ \ext^*_{A^e}(A,A) \xrightarrow{\eta_A} \ext^*_{A^e}(A,\G\F(A)) \cong \ext^*_{A^l}(\F(A),\F(A)) \cong \ext^*_{A_l}(k,k) , \]
where the left isomorphism comes from the adjunction. As above, when $A$ is simply connected the map $\alpha$ does not depend on our choice of isomorphism $\F(A)\cong k$. From classical category theory we learn that the map $\ext^*_{A^e}(A,A) \xrightarrow{\F} \ext^*_{A^l}(\F(A),\F(A))$ is equal to the composition $\ext^*_{A^e}(A,A) \xrightarrow{\eta_A} \ext^*_{A^e}(A,\G\F(A)) \xrightarrow{\cong} \ext^*_{A^l}(\F(A),\F(A))$. Hence we conclude that $\alpha$ and $\shr$ are equal (for the same isomorphism $\varphi$).

\subsection{The unit map}
For what follows we need to identify the unit map $\eta_A$. First we define a map $\mu:k \ba k \to k$ to be the composition of the natural map $k \ba k \to k \otimes_A k$ from Lemma~\ref{lem: Map from derived A-tensor to non derived} with the isomorphism $k \otimes_A k \cong k$. One can think of $\mu$ as extending the pairing $\mu_{p,q}$ of Definition~\ref{def: The pairing mu}.

Next we choose a model for the derived functor of $F$, set
\[\mathbf{L}F(M)=F(M\ba A) \cong M\ba k , \]
we leave it to the reader to ascertain this indeed yields the derived functor of $F$. Now $\mathbf{L}F(A)=F(A\ba A)$ and $GF(A\ba A)=A\ba k$ and the unit map $A\ba A \to GF(A\ba A)$ is clearly $1 \ba \aug$. Commutation of the following diagram
\[ \xymatrixcompile{
{A\ba A} \ar[r]^{1\ba \aug} \ar[d]^\psi & {A\ba k} \ar[d]^{\mu(\aug \ba 1)} \\
{A} \ar[r]^\aug & {k}
}\]
shows that $\shr$ is equal to the composition
\[ \ext^*_{A^e}(A,A) \xrightarrow{\aug} \ext^*_{A^e}(A,k) \cong \ext^*_{A_l}(k,k) . \]
Since the map $I:\ext^*_{A^e}(A,A) \to \ext^*_{A_l}(k,k)$ defined by \cite{FelixThomasVPHochschild} is induced by $\ext_{A^e}^*(A,\aug)$ we have proven
\begin{lemma}
\label{lem: Sheering map as in FTVP paper}
The shearing map $\shr:\ext^*_{A^e}(A,A) \to \ext^*_{A_l}(k,k)$ is equal to the morphism $I$ defined in \cite{FelixThomasVPHochschild}.
\end{lemma}

In addition we can now identify the morphism adjoint to a given element $g\in \ext^*_{A_l}(k,k)$. The map $\mu(\aug \ba 1):A \ba k \to k$, is clearly a cofibrant replacement of $k$. We represent $g$ by the map $\tilde{g}: A\ba k \to \Sigma^n k$. By classical category theory the adjoint map to $\tilde{g}$ is $g(1\ba \aug):A\ba A \to \Sigma^n k$.

\subsection{The adjunction}
To specify the adjunction we need to identify the counit map $\epsilon_k:\F\G(k) \to k$ as well.
\begin{lemma}
The counit $\epsilon_k:\F\G(k) \to k$ map is represented by $\mu: k \ba k \to k$.
\end{lemma}
\begin{proof}
The counit map is the composition $\mathbf{L}F (G(k)) \xrightarrow{l_{G(k)}} FG(k) \xrightarrow{\varepsilon_k} k$ where $l$ is the natural map $\mathbf{L}F \to F$ and $\varepsilon$ is the counit for the adjunction of $F$ and $G$. Now $\varepsilon_k$ is simply the isomorphism $k\otimes k \cong k$. From the definition it is easy to see that $l_{G(k)}$ is $\mu$.
\end{proof}

\begin{corollary}
\label{cor: The adjunction}
\begin{enumerate}
\item Let $f \in \ext^*_{A^e}(A,k)$ be represented by $\tilde{f}:\tilde{A} \to \Sigma^n k$, then the adjoint is represented by $\mu (\tilde{f} \ba k)$.
\item Let $g \in \ext^*_{A_l}(k,k)$ be represented by $\tilde{g}:A\ba k \to \Sigma^n k$, then the adjoint is represented by $g(1\ba \aug)$.
\end{enumerate}
\end{corollary}

\subsection{The induced pairing}
Since $\ext^*_{A^e}(A,k)$ is isomorphic to $\ext^*_{A_l}(k,k)$, there is a graded algebra structure on $\ext^*_{A^e}(A,k)$. We detail this structure.

\begin{definition}
Define a pairing $m:\ext^i_{A^e}(A,k) \otimes \ext^j_{A^e}(A,k) \to \ext^{i+j}_{A^e}(A,k)$ in the following way. Let $\tilde{f}:\tilde{A} \to \Sigma^i k$ represent an element $f\in\ext^i_{A^e}(A,k)$ and let $\tilde{g}:\tilde{A} \to \Sigma^j k$ represent an element $g\in\ext^j_{A^e}(A,k)$. Let $m(f\otimes g)$ be the morphism represented by the composition
\[ \xymatrixcompile{
{\tilde{A}} \ar[r]^-{\sim} & {\tilde{A} \ba \tilde{A}} \ar[r]^-{\tilde{f} \ba \tilde{g}} & {\Sigma^i k \ba \Sigma^j k} \ar[r]^-{\mu} & \Sigma^{i+j} k . } \]
\end{definition}

\begin{lemma}
\label{lem: The induced pairing}
Under the isomorphism $\ext^*_{A^e}(A,k)\cong \ext^*_{A_l}(k,k)$ given by the adjunction, the pairing $m$ corresponds to the multiplication on $\ext^*_{A_l}(k,k)$.
\end{lemma}
\begin{proof}
By Corollary~\ref{cor: The adjunction} we need first to represent the composition of $\mu(f\ba k)$ with $\mu(g\ba k)$. One easily sees that this composition is represented by
\[(*) \quad {\tilde{A} \ba \tilde{A} \ba k} \xrightarrow{1 \ba \tilde{g} \ba 1} {\tilde{A} \ba \Sigma^j k \ba k} \xrightarrow{1\ba \mu} {\tilde{A} \ba \Sigma^j k} \xrightarrow{\tilde{f} \ba 1} {\Sigma^i k \ba \Sigma^j k} \xrightarrow{\mu} \Sigma^{i+j} k . \]
This composition is equal to $\mu(1\ba \mu)(\tilde{f} \ba \tilde{g} \ba \tilde 1_k)$. It is easy to see that $\mu$ satisfies the associativity relation $\mu(1\ba \mu) = \mu(\mu \ba 1)$. This, together with the associativity of $\ba$, shows that the composition $(*)$ is equal to $\mu(\mu(\tilde{f} \ba \tilde{g}) \ba 1_k)$, whose adjoint is $\mu(\tilde{f} \ba \tilde{g})$ by Corollary~\ref{cor: The adjunction}.
\end{proof}

\subsection{The third description of the shearing map}
\label{sub: The third description of the shearing map}
Every element $x \in HH^n(A)$ induces a natural transformation of functors $\zeta(x):1_{\Derived(A)} \to \Sigma^{n}1_{\Derived(A)}$ which we now describe. For an $A$-module $M$ the morphism $\zeta(x)_M$ is represented by
\[ M \xleftarrow{u} \tilde{A} \ba M \xrightarrow{\tilde{x} \ba M} \Sigma^n \tilde{A} \ba M \xrightarrow{u} M  \]
where $\tilde{x}:\tilde{A} \to \Sigma^n \tilde{A}$ represents $x$ and $u=e(e\ba 1_M)$. Addition and multiplication of elements in $HH^*(A)$ become addition and composition of such natural transformations. In this way we get a map $\zeta$ of graded commutative rings from $HH^*(A)$ to the graded commutative ring of natural transformations $\{1_{\Derived(A)} \to \Sigma^{n}1_{\Derived(A)} \}_n$. This is a well known map whose target is called the \emph{centre of $\Derived(A)$}.

We can now define a third map $HH^*(A) \to \ext^*_A(k,k)$ by $x \mapsto \zeta(x)_k$. It is easy to see this map is the shearing map as defined in Definition~\ref{def: Sheering map}. From this description it is also clear that the image of the shearing map lies in the graded commutative centre of $\ext^*_A(k,k)$.

\section{Identifying the $E^2$-term}
\label{sec: Identifying the $E^2$-term}

To identify the $E^2$-term of the spectral sequence we must assume $A$ is simply connected. As we shall see, this assumption gives the differential on the $E^1$-term a recognizable form.

\subsection{Simple connectedness}
The augmentation $A^e \to k$ induces adjoint functors
\[F:\Derived(A^e) \leftrightarrows \Derived(k):G .\]
By a result of Dwyer, Greenlees and Iyengar~\cite[Proposition 3.9]{DwyerGreenleesIyengar} the bimodule $A^n$ is isomorphic in $\Derived(A^e)$ to $G(V)$ for some $k$-vector space $V$. But this isomorphism is not natural. Worse still the morphisms $\delta=\delta^n:A^n \to A^{n+1}$ given by the composition $\theta\kappa$ might not be induced by morphisms in $\Derived(k)$. However, when $A$ is simply connected we can show that the morphisms $\delta^n$ are in the image of $G$. We start with the following well known property.

\begin{lemma}
If $A$ is simply connected then so is $A^e$.
\end{lemma}
\begin{proof}
Let $F$ be a cofibrant $A^e$-module and let $M$ be a left $A$-module and $N$ be a right $A$-module. Hence $N \otimes_k M$ is a right $A^e$-module. It is easy to see there is an isomorphism
\[ (N \otimes_k M) \otimes_{A^e} F \cong N\otimes_A F \otimes_A M .\]
Now let $P$ be a cofibrant replacement of $k$ as a left $A$-module and let $Q$ be a cofibrant replacement of $k$ as a right $A$-module. Then $P \otimes_k Q$ is a cofibrant replacement of $k$ as an $A^e$-module and
\[ k \otimes_{A^e} P\otimes_k Q \cong (k \otimes_k k) \otimes_{A^e} P\otimes_k Q \cong (k\otimes_A P) \otimes_k (Q \otimes_A k) . \]
We conclude that $\tor_*^{A^e}(k,k) \cong \tor_*^{A}(k,k) \otimes_k \tor_*^{A}(k,k)$ (the tensor product of graded vector spaces). In particular $\tor_0^{A^e}(k,k)=k$ if and only if $\tor_0^{A}(k,k)=k$. For any coconnective augmented dga $B$ there is an isomorphism $\ext^*_{B}(k,k) \cong \hom_k(\tor_{-*}^{B}(k,k),k)$, which completes the proof.
\end{proof}

\begin{lemma}
\label{lem: Simple connectedness}
For every $n$ the morphism $\delta^n:A^n \to A^{n+1}$ is isomorphic to $G(\epsilon^n)$ for some morphism $\epsilon^n$ in $\Derived(k)$. The morphism $\epsilon^n$ is itself isomorphic to the $n$'th differential of $A$.
\end{lemma}
\begin{proof}
The left adjoint $F:\Derived(A^e) \to \Derived(k)$ is given by $k \otimes^\mathbf{L}_{A^e} -$. Since $A^n$ is isomorphic to a direct sum of copies of $k$ in $\Derived(A^e)$ we see that
\[H^0(F(A^n)) \cong \tor_0^{A^e}(k,A^n) \cong A^n .\]
Note that this is only an isomorphism of $k$-vector spaces. To avoid confusion let $V^n$ be a $k$-vector space isomorphic to $H^0(F(A^n))$. Then $V^n$ is a direct summand of $F(A^n)$. Moreover, it is easy to see that the composition $A^n \to GF(A^n) \to G(V^n)$ is an isomorphism in $\Derived(A^e)$.

The morphism $\delta^n$ induces $F(\delta^n):F(A^n) \to F(A^{n+1})$. Utilizing the fact that every object in $\Derived(k)$ is isomorphic to a direct sum of its homology groups one gets a morphism $\epsilon^n:V^n \to V^{n+1}$ which commutes with $F(\delta^n)$. Applying the functor $G$ yields a commutative diagram
\[ \xymatrixcompile{
{A^n} \ar[d]^{\delta^n} \ar[r] & {GF(A^n)} \ar[d]^{GF\delta^n} \ar[r] & {G(V^n)} \ar[d]^{G\epsilon^n} \\
{A^{n+1}} \ar[r] & {GF(A^{n+1})} \ar[r] & {G(V^{n+1})} \\
}\]
which completes the proof.
\end{proof}

\subsection{The $E^2$-term of the spectral sequence}
We start with the general case.

\begin{lemma}
\label{lem: The general E2-term}
Let $A$ be a simply connected $k$-dga. The $E^2$-term of the spectral sequence is
\[ E^2_{p,q} \cong \ext_{A^e}^{-q}(A,H^{-p}(A)) . \]
If $A$ has bounded homology then the spectral sequence has strong convergence.
\end{lemma}
\begin{proof}
>From Lemma~\ref{lem: Simple connectedness} we see that the morphism $\delta^n:A^n \to A^{n+1}$ is equal to $G(\epsilon^n)$ for some morphism $\epsilon^n:V^n \to V^{n+1}$ in $\Derived(k)$, where $V^i$ is isomorphic to $A^i$ as vector spaces. One can therefore decompose $V^0$ as $H^0(A)\oplus \epsilon(V^0)$ and continuing inductively write $V^n$ as $\epsilon(V^{n-1})\oplus H^n(A) \oplus \epsilon(V^n)$. Applying the functor $G$ one sees that $A^n$ is isomorphic to $G\epsilon(V^{n-1})\oplus H^n(A) \oplus G\epsilon(V^n)$ and that these isomorphisms are compatible with the morphisms $\delta^n$.

Recall that the $E^1$-term of the spectral sequence is the cochain-complex of graded groups $E^1_{p,*} = \ext^{-q}_{A^e}(A,A^{-p})$ with the differential induced by $\delta$. The decompositions of $A^n$ imply that the homology of this cochain-complex is isomorphic to
\[ \ext_{A^e}^{-q}(A,H^{-p}(A)) , \]
which is therefore the $E^2$-term.

Assuming that $H^*(A)$ is bounded implies that the $E^2_{p,q}=0$ whenever $p>0$ or $p<n$ for some fixed $n$. Hence the spectral sequence collapses at some finite stage $r$. By the remark following \cite[Theorem 7.1]{BoardmanCCSpecSeq} we see that the spectral sequence has strong convergence
\end{proof}

\begin{lemma}
\label{lem: Identifying the E2-term}
Suppose $A$ is simply connected of finite type, then there is an isomorphism of graded algebras
\[ E^2_{p,q} \cong H^{-p}(A)\otimes_k \ext^{-q}_A(k,k) . \]
\end{lemma}
\begin{proof}
Let $B$ be the differential bigraded algebra given by
\[ B_{p,q} = A^{-p} \otimes_k \ext^{-q}_A(k,k) \ \text{ and } \ d^B(a \otimes f) = d^A(a) \otimes f \]
where $d^A$ is the given differential on $A$. We will show there exists a morphism of differential bigraded algebras $\phi:B \to E^1$ which induces an isomorphism $H_*(\phi):H_*(B) \to H_*(E^1)=E^2$.

Composition of morphisms yields natural morphisms
\[ \ext_{A^e}^0(k,A^n) \otimes_k \ext_{A^e}^i(A,k) \to \ext_{A^e}^i(A,A^n) . \]
There is an obvious monomorphism $A^n \to \ext_k^0(k,A^n) \to \ext_{A^e}^0(k,A^n)$. Note that since $A^e$ is simply connected, the morphism $\ext_k^0(k,A^n) \to \ext_{A^e}^0(k,A^n)$ is in fact an isomorphism, because:
\[\ext_{A^e}^0(k,A^n) \cong \ext_{A^e}^0(k,G(V^n)) \cong \ext^0_{k}(k \ba k, V^n) \cong \ext^0_k(k,V^n) . \]
In this way we get morphisms
\[ \lambda_{p,q}:A^{-p} \otimes_k \ext_{A^e}^{-q}(A,k) \to \ext_{A^e}^{-q}(A,A^{-p}) . \]
Recall that there is an isomorphism $\ext_A^i(k,k) \cong \ext_{A^e}^i(A,k)$. Thus we have a morphism of graded vector spaces:
\[\phi_{p,q}: A^{-p} \otimes_k \ext^{-q}_A(k,k) \to E^1_{p,q}=\ext^{-q}_{A^e}(A,A^{-p}) .\]
Since the differential on $E^1$ is simply $\ext_{A^e}^*(A,d^A)$ we see that $\phi$ yields a morphism of differential bigraded vector spaces $\phi:B \to E^1$.

Next we show that $\phi$ is a quasi-isomorphism. From Lemma~\ref{lem: The general E2-term} we see that $H\phi$ is the morphism
\[ H^n(A) \otimes \ext^m_{A}(k,k) \to \ext^m_{A^e}(A,H^n(A)) .\]
Since $H^n(A)$ is isomorphic in $\Derived(A^e)$ to a finite coproduct of copies of $k$, this morphism is an isomorphism.

It remains to show that $\phi$ is a morphism of graded algebras. Let $\tilde{k}$ be a cofibrant replacement of $k$ over $A$. Suppose given elements $x \in A^{-t}$, $u \in A^{-s}$ and maps $y:\tilde{k} \to \Sigma^n k$ and $v:\tilde{k} \to \Sigma^m k$. We shall now compute the product $\phi(x\otimes y) \cdot \phi(u\otimes v)$.

Choose maps $x:k \to A^{-t}$ and $u: k \to A^{-s}$ representing the elements $x$ and $y$. Recall that $\tilde{A}$ is a cofibrant replacement for $A$ over $A^e$ and let $a:\tilde{A} \to \tilde{k}$ be a map equivalent to the augmentation map $\aug:A \to k$. From the definition $\phi$ and Corollary~\ref{cor: The adjunction} one sees that $\phi(x\otimes y)$ is represented by the composition $xya$. The product $\phi(x\otimes y) \cdot \phi(u\otimes v)$ is then represented by the composition
\[ (*) \quad \xymatrixcompile{
{\tilde{A}} \ar[r]^-{\sim} & {\tilde{A} \ba \tilde{A}} \ar[r]^-{ay\ba av} & {\Sigma^n k\ba \Sigma^m k} \ar[r]^-{x \ba u} & {\Sigma^n A^{-t} \ba \Sigma^m A^{-s} } \ar[r]^-{\mu_{s,t}} & {\Sigma^{n+m} A^{-s-t}} . } \]
Let $w:k \to A^{-s-t}$ be a map representing the element $xu \in A^{-s-t}$, i.e. $w(1)=xu$. We leave it to the reader to ascertain that the composition $(*)$ above is equal to the composition
\[\xymatrixcompile{
{\tilde{A}} \ar[r]^-{\sim} & {\tilde{A} \ba \tilde{A}} \ar[r]^-{ay\ba av} & {\Sigma^n k\ba \Sigma^m k} \ar[r]^-{\mu} & {\Sigma ^{n+m} k} \ar[r]^-{w} & {\Sigma^{n+m} A^{-s-t}} . } \]

>From Lemma~\ref{lem: The induced pairing} we see that, under the isomorphism $\ext^*_{A^e}(A,k) \cong \ext^*_A(k,k)$, the composition
\[\xymatrixcompile{
{\tilde{A}} \ar[r]^-{\sim} & {\tilde{A} \ba \tilde{A}} \ar[r]^-{ay\ba av} & {\Sigma^n k\ba \Sigma^m k} \ar[r]^-{\mu} & {\Sigma ^{n+m} k} } \]
corresponds to the composition $y \circ v \in \ext^*_A(k,k)$. We conclude that
\[ \phi^{-1} (\phi(x\otimes y) \cdot \phi(u\otimes v)) = xu \otimes (y\circ v) . \]
\end{proof}

\begin{lemma}
\label{lem: Infinite cycles}
Suppose $A$ is simply connected. Then under the isomorphism $E^2_{0,*} \cong \ext^*_A(k,k)$, the infinite cycles in $E^2_{0,*}$ can be identified with the image of the shearing map $\shr:HH^*(A) \to \ext^*_A(k,k)$.
\end{lemma}
\begin{proof}
Consider the trivial filtration on $k$ given by $I(0)=k$ and $I(p)=0$ for $p>0$. As the filtration of $A$ gave a spectral sequence for $\ext_{A^e}(A,A)$ in Section~\ref{sec: Construction of the spectral sequence}, so does this filtration of $k$ yield a spectral sequence $E'^r_{p,q}$ for $\ext_{A^e}^*(A,k)$, one which collapses at the $E'^1$ stage. Indeed  $E'^1_{0,*} = \ext_{A^e}^*(A,k)$ and is zero otherwise.

The augmentation map $\aug:A \to k$ is then a map of filtered $A$-bimodules, and so induces a map of the aforementioned spectral sequences. This map is easily described on the $E^2$-term: it is the identity map $E^2_{0,*} \to E'^2_{0,*}$ and zero elsewhere. It follows that the infinite cycles in $E^2_{0,*}$ are the image of the map $\ext^*_{A^e}(A,A) \to \ext^*_{A^e}(A,k)$ induced by the augmentation $\aug:A \to k$. As we saw in~\ref{sub: Two descriptions of the shearing map} this indeed gives the image of the shearing map after composing with the adjunction isomorphism $\ext^*_{A^e}(A,k) \cong \ext^*_A(k,k)$.
\end{proof}

\bibliographystyle{amsplain}    
\bibliography{bib2010}          

\end{document}